\documentclass[11pt,reqno]{amsart}
\usepackage{hyperref}

\usepackage{amsmath}
\usepackage{mathabx}
\usepackage{subcaption}
\usepackage{graphicx}

\begin{document}

\title[   OCP for the CH Equation]
{Solution of the Optimal Control Problem for the Cahn-Hilliard Equation Using Finite Difference Approximation}
 
\author[G. Garai. \& B. C. Mandal ]
{Gobinda Garai \& Bankim C. Mandal}

\address{Gobinda Garai \newline
School of Basic Sciences,
Indian Institute of Technology Bhubaneswar, India}
\email{gg14@iitbbs.ac.in}

\address{Bankim C. Mandal \newline
School of Basic Sciences,
Indian Institute of Technology Bhubaneswar, India}
\email{bmandal@iitbbs.ac.in} 

\thanks{Submitted.}
\subjclass[]{65M06, 65M12, 49M41}
\keywords{Optimal control problem, Constrained Optimization,  Convergence analysis,  Cahn-Hilliard equation.}

\begin{abstract}
This paper is concerned with the designing, analyzing and implementing linear and nonlinear discretization scheme for the distributed optimal control problem (OCP) with the Cahn-Hilliard (CH) equation as constrained. We propose three difference schemes to approximate and investigate the solution behaviour of the OCP for the CH equation. We present the convergence analysis of the proposed discretization.  We verify our findings by presenting numerical experiments.
\end{abstract}

\maketitle
\numberwithin{equation}{section}
\newtheorem{theorem}{Theorem}[section]
\newtheorem{lemma}[theorem]{Lemma}
\newtheorem{definition}[theorem]{Definition}
\newtheorem{proposition}[theorem]{Proposition}
\newtheorem{remark}[theorem]{Remark}
\allowdisplaybreaks

\section{Introduction} \label{intro}
In this work, we consider the following distributed optimal control problem (OCP)
\begin{equation}\label{OCP}
\min\limits_{u(x,t)\in L^2(0, T; L^2(\Omega))} J(y,u)=\frac{1}{2}\int_{0}^T\int_{\Omega} \left( y(x,t)- \widehat{y}(x,t)\right) ^2 dx dt + \frac{\lambda}{2}\int_0^T\int_{\Omega}  u(x,t)^2 dx dt,
\end{equation}
subject to the Cahn-Hilliard (CH) equation 
\begin{equation}\label{CH_state}
\begin{cases}
\frac{\partial y}{\partial t} = \Delta f(y) - \epsilon^2\Delta^2 y + u, & (x,t)\in\Omega\times(0,T],\\
\partial_{\nu}y=0=\partial_{\nu}(\Delta y), & (x,t)\in\partial\Omega\times(0,T],\\
u(x, 0)=u_0, 
\end{cases} 
\end{equation}
where $\Omega\subset \mathbb{R}^d (d=1, 2)$ is a bounded Lipschitz domain,  $\nu$ is the unit outward normal, $u\in L^2(0, T; L^2(\Omega))$ is the control variable, $\widehat{y}\in L^2(0, T; L^2(\Omega))$ is the target state, $\lambda>0$ is a regularization parameter, $0< \epsilon\ll 1$ and $f(y)=y^3-y$. The purpose of the optimization is to determine the control function $u(x,t)$ in such a way that the state $y(x,t)$ gets as closely as possible to a given desired state $\widehat{y}(x,t)$. Here we follow optimize-then-discretize approach to discretize the OCP \eqref{OCP} - \eqref{CH_state}.

The existence and uniqueness of OCP \eqref{OCP} - \eqref{CH_state} follows from the standard variational arguments \cite{hinze2008optimization, troltzsch2010optimal}, which can be described by its first order optimality system, see [Troltzsch].  The first order optimality system yields the state equation \eqref{CH_state}, the adjoint equation 
\begin{equation}\label{CH_adjoint}
\begin{cases}
-\frac{\partial p}{\partial t} = f'(y)\Delta p - \epsilon^2\Delta^2 p + \widehat{y}-y, & (x,t)\in\Omega\times(0,T],\\
\partial_{\nu}p=0=\partial_{\nu}(\Delta p), & (x,t)\in\partial\Omega\times(0,T],\\
p(x, T)=0, 
\end{cases} 
\end{equation}
and the optimality condition
\begin{equation}\label{CH_optimality}
p(x,t)=\lambda u(x,t),  (x,t)\in\Omega\times(0,T].
\end{equation}

The CH equation, represented by Equation (1) without the control term $u$, is a mathematical model commonly used to describe the evolution of a binary melted alloy below its critical temperature as referenced in \cite{Cahn,Hilliard}. Numerous research studies have been devoted to developing numerical schemes for approximating the solution of the CH equation. These schemes often employ finite difference or finite element methods, with either Dirichlet boundary conditions \cite{DuNicolaides, David} or Neumann boundary conditions \cite{elliott1987numerical, furihata2001stable, shin2011conservative, ElliottZheng, Elliott, StuartHumphries, Christlieb, ChengFeng}, among others. For further insights into the numerical approaches used for solving the CH equation can be found in \cite{lee2014physical}.

The solution of the CH equation encompasses two separate dynamics: phase separation, characterized by rapid changes over time, and phase coarsening, which occurs at a slower pace. During the initial stages, the formation of fine-scale phase regions takes place, with interfaces of width $\epsilon$ separating them. On the other hand, during phase coarsening, the solution tends to converge towards an equilibrium state. The OCP associated with the CH equation becomes crucial in situations where there is a desire to exert influence over the phase separation or coarsening behaviour, or to attain specific concentration profiles. Within this context, the OCP plays an important role by providing a framework to actively control and shape the dynamics of phase separation or coarsening in order to achieve desired outcomes.
The following assumptions are made regarding the nonlinear terms:
\begin{equation}\label{Lipchitz_CH}
\max_y \vert f'(y)\vert \leq M,\; \max_y \vert \widetilde{f}'(y)\vert \leq \widetilde{M}  
\end{equation}
where $M, \widetilde{M}$ are non-negative constant and $\widetilde{f}(y)=y^3$. 

We now talk about a few studies that relate OCP to the CH equation. The authors of \cite{yong1991feedback, hintermuller2012distributed} took into account the OCP for the CH equation and presented the existence results. The studies in \cite{zhao2011optimal, duan2015optimal} are concerned with OCP for the viscous CH equation and their existence results. You may see a study for OCP with convective CH as constraints in \cite{zhao2014optimal}. The OCP for the CH equation with state constraint has been studied in \cite{zheng2015optimal}. However, the above stated body of works do not have any numerical scheme to solve the OCP. This is, as far as we are aware, the first attempt at constructing a numerical framework for the OCP of the CH equation. In order to discretize continuous spatial and temporal variables, we employ the finite difference approach.

Rest of the paper is organised as follows. In Section \ref{Section2} we present the numerical schemes to approximate the solution. In Section \ref{Section3} we show the convergence results of the proposed schemes. Finally we present numerical results in Section \ref{section4}.

\section{Notation and Difference Schemes}\label{Section2}
We propose the finite difference scheme for \eqref{CH_state}-\eqref{CH_adjoint} in one spatial dimension, while an extension to higher spatial dimension is straight forward. To discretize the problem \eqref{CH_state}-\eqref{CH_adjoint} in the spatial domain $\Omega=(a, b)$, we partitioned the domain uniformly by taking mesh size $h=(b-a)/(N-1)$, where $N$ is the number of spatial grid point. Then the discrete domain is $\Omega_h=\{(x_{i-1},x_i): x_i=a+(i-1)h, i=1,2,...,N\}.$ Let $\delta_t=T/N_t$ be the uniform time step corresponding to discrete temporal variable $t_n=n\delta_t$ for $n=0,1,...,N_t$, where $N_t$ is any positive integer. For any function $z$ define on $\Omega_h$, we denote $z_i=z(x_i)$ and define the following difference operators 
\[ \nabla_h^{+}z_i=\frac{z_{i+1}-z_i}{h},\; \nabla_h^{-}z_i=\frac{z_{i}-z_{i-1}}{h},\; \Delta_h=\nabla_h^{+}\nabla_h^{-}=\nabla_h^{-}\nabla_h^{+},\; \Delta_h^2=\Delta_h\Delta_h.\]
Now we define the discrete $L^2$-inner product as 
$(y, z)_h=\sum\limits_{i=1}^{N} y_i z_i$ for any function $y, z$,  defined on $\Omega_h$. Then the corresponding discrete $L^2$-norm is given by $\| z \|_h=\sqrt{(z, z)_h}$. We also define discrete $H^1$-seminorm, $H^2$-seminorm and maximum-norm for any grid function $z$ in $\Omega_h$ as 
\[
\vert z\vert_{1,h}^2= h \sum\limits_{i=1}^{N} (\nabla_h^{-}z_i)^2, \;  \vert z\vert_{2,h}^2= h \sum\limits_{i=1}^{N} (\Delta_h z_i)^2, \; \| z \|_{\infty, h}= \sup\limits_i \vert z_i \vert. 
\]
Let $Y_i^n = Y(x_i, t_n), P_i^n = P(x_i, t_n)$ and $U_i^n=U(x_i, t_n)$ for $i=1,2,...,N$ and $n=0,1,...,N_t$. Then to get the approximate solution of \eqref{CH_state}-\eqref{CH_adjoint} we discretize the state equation \eqref{CH_state} as
\begin{equation}\label{discrete_state}
\frac{Y_i^{n+1}-Y_i^n}{\delta_t}=\Delta_h f(Y_i^{n+1})-\epsilon^2 \Delta_h^2Y_i^{n+1} + \frac{1}{\lambda}P_i^{n},
\end{equation} 
and the adjoint equation \eqref{CH_adjoint} as
\begin{equation}\label{discrete_adjoint}
-\frac{P_i^{n+1}-P_i^n}{\delta_t}=f'(Y_i^{n})\Delta_h P_i^n -\epsilon^2 \Delta_h^2P_i^{n} + \widehat{Y}_i^{n+1}-Y_i^{n+1}.
\end{equation} 
Observe that we made a substitution of control variable $U_i^n$ by $\frac{1}{\lambda}P_i^{n}$ in \eqref{discrete_state} using the relation in  \eqref{CH_optimality}. The described scheme  \eqref{discrete_state} - \eqref{discrete_adjoint} is undoubtedly nonlinear, with the nonlinearity originating from \eqref{discrete_state}.  We call the scheme \eqref{discrete_state} - \eqref{discrete_adjoint} as $\textbf{S}1$. Next we display two possible linear approximation to the state equation \eqref{CH_state}. First we consider the following linear scheme for the state equation 
\begin{equation}\label{discrete_state2}
\frac{Y_i^{n+1}-Y_i^n}{\delta_t}=\Delta_h ((Y_i^{n})^2 Y_i^{n+1}) -\Delta_h Y_i^n-\epsilon^2 \Delta_h^2Y_i^{n+1} + \frac{1}{\lambda}P_i^{n}.
\end{equation} 
We call the state approximation scheme \eqref{discrete_state2}, along with adjoint approximation scheme \eqref{discrete_adjoint} as $\textbf{S}2$. Another way to discretize the state equation is 
\begin{equation}\label{discrete_state3}
\frac{Y_i^{n+1}-Y_i^n}{\delta_t}=\Delta_h (Y_i^{n})^3  -3\Delta_h Y_i^n + 2\Delta_h Y_i^{n+1}-\epsilon^2 \Delta_h^2Y_i^{n+1} + \frac{1}{\lambda}P_i^{n}.
\end{equation} 
We call the state approximation scheme \eqref{discrete_state3}, along with adjoint approximation scheme \eqref{discrete_adjoint} as $\textbf{S}3$.
\section{Convergence of the Schemes}\label{Section3}
We examine the convergence of the finite difference scheme \eqref{discrete_state}-\eqref{discrete_adjoint} in this section. Discrete Gronwall's inequality, which will be used extensively in our convergence estimations, was previously discussed in \cite{heywood1990finite}.
\begin{lemma}[Discrete Gronwall inequality]
Let $\tau, Q$ and $a_n, b_n, c_n, d_n$ be non-negative numbers for integers $n\geq 0$ such that
\[
a_n + \tau\sum\limits_{k=0}^n b_k \leq \tau\sum\limits_{k=0}^n d_k a_k + \tau\sum\limits_{k=0}^n c_k + Q.
\]
Suppose that $\tau d_k<1, \forall k$ and set $\sigma_k=1/(1-\tau d_k)$. Then the following holds for $n\geq 0$
\[
a_n + \tau\sum\limits_{k=0}^n b_k \leq \exp \left( \tau\sum\limits_{k=1}^n \sigma_k d_k\right) \left[ \tau\sum\limits_{k=0}^n c_k +Q \right].
\]
\end{lemma}
\begin{theorem}[Convergence of the scheme $\textbf{S}1$]\label{thm1_ocp}
Let $y(x,t)$ and $p(x,t)$ be sufficiently smooth functions. For $\delta_t$ sufficiently small, the finite difference scheme \eqref{discrete_state}-\eqref{discrete_adjoint} is first order in time and second order in space convergent, i.e., 
$$  \max\limits_n \left\{\parallel  y^n-Y^n \parallel_{\infty,h} + \parallel p^n-P^n \parallel_{\infty, h} \right\}  \leq C ( \delta_t + h^2 ).$$
\end{theorem}
\begin{proof}
Using Taylor expansion we observe that the exact solution $y_i^n=y(x_i, t_n)$ and $p_i^n=p(x_i, t_n)$ satisfy the following equations
\begin{equation}\label{taylor_state}
\frac{y_i^{n+1}-y_i^n}{\delta_t}=\Delta_h f(y_i^{n+1})-\epsilon^2 \Delta_h^2y_i^{n+1} + \frac{1}{\lambda}p_i^{n} + F_i^n,
\end{equation} 
\begin{equation}\label{taylor_adjoint}
-\frac{p_i^{n+1}-p_i^n}{\delta_t}=f'(y_i^{n})\Delta_h p_i^n -\epsilon^2 \Delta_h^2p_i^{n} + \widehat{y}_i^{n+1}-y_i^{n+1} + G_i^n,
\end{equation} 
where $F_i^n$ and $G_i^n$ denote the truncation error, which satisfy the following for some positive constants $c_1, c_2$
 \begin{equation}\label{truncation_err}
\max\limits_{i,n} \vert F_i^n\vert \leq c_1 ( \delta_t + h^2 ),\;\; \max\limits_{i,n} \vert G_i^n\vert \leq c_2 ( \delta_t + h^2 ).  
 \end{equation}
Let us define the error $e_i^n=y_i^n-Y_i^n$ and $\mathfrak{e}_i^n=p_i^n-P_i^n$. Then taking difference between \eqref{taylor_state}-\eqref{taylor_adjoint} and \eqref{discrete_state}-\eqref{discrete_adjoint} yields
\begin{equation}\label{err_state}
\frac{e_i^{n+1}-e_i^n}{\delta_t}=\Delta_h \left(f(y_i^{n+1})-f(Y_i^{n+1})\right)-\epsilon^2 \Delta_h^2e_i^{n+1} + \frac{1}{\lambda}\mathfrak{e}_i^{n} + F_i^n,
\end{equation} 
\begin{equation}\label{err_adjoint}
-\frac{\mathfrak{e}_i^{n+1}-\mathfrak{e}_i^n}{\delta_t}=f'(y_i^{n})\Delta_h p_i^n -f'(Y_i^{n})\Delta_h P_i^n -\epsilon^2 \Delta_h^2\mathfrak{e}_i^{n} -e_i^{n+1} + G_i^n.
\end{equation} 
Taking the inner product of \eqref{err_state} and $e^{n+1}$ yields
\begin{equation}\label{est_state1}
\left(\frac{e^{n+1}-e^n}{\delta_t}, e^{n+1}\right)_h=\left(\Delta_h \left(f(y^{n+1})-f(Y^{n+1})\right), e^{n+1}\right)_h-\epsilon^2 \vert e^{n+1}\vert_{2,h}^2 + \frac{1}{\lambda}\left(\mathfrak{e}^{n}, e^{n+1}\right)_h + \left(F^n, e^{n+1}\right)_h.
\end{equation}
Using the differentiability of $f$, Cauchy-Schwarz inequality and $\left(\frac{e^{n+1}-e^n}{\delta_t}, e^{n+1}\right)_h \geq \frac{\| e^{n+1}\|_h^2 - \| e^{n}\|_h^2}{2\delta_t}$ on \eqref{est_state1} we obtain
\begin{equation}\label{est_state2}
\frac{\| e^{n+1}\|_h^2 - \|e^{n}\|_h^2}{2\delta_t} \leq M \| e^{n+1}\|_h \vert e^{n+1}\vert_{2,h}-\epsilon^2 \vert e^{n+1}\vert_{2,h}^2 + \frac{1}{\lambda}\| \mathfrak{e}^{n}\|_h \|e^{n+1}\|_h + \| F^{n}\|_h\| e^{n+1}\|_h.
\end{equation}
Using Young's inequality and \eqref{truncation_err} on \eqref{est_state2} we get
\begin{equation}\label{est_state3}
\frac{\| e^{n+1}\|_h^2 - \| e^{n}\|_h^2}{2\delta_t} \leq \underbrace{\left(\frac{M^2}{4\epsilon^2} + \frac{1}{2\lambda} + \frac{1}{2} \right)}_{\alpha} \| e^{n+1}\|_h^2 + \frac{1}{2\lambda}\| \mathfrak{e}^{n}\|_h^2 + \frac{c_1}{2} ( \delta_t + h^2 )^2.
\end{equation}
From \eqref{est_state3} we have 
\begin{equation}\label{est_state4}
\left( 1-2\delta_t\alpha\right) \parallel e^{n+1}\parallel_h^2 \leq \parallel e^{n}\parallel_h^2 + \frac{\delta_t}{\lambda}\parallel \mathfrak{e}^{n}\parallel_h^2 + c_1\delta_t ( \delta_t + h^2 )^2.
\end{equation}
For $\delta_t<\frac{1}{2\alpha}$ in \eqref{est_state4} and taking the sum over $n$ we have 
\begin{equation}\label{est_state5}
\parallel e^{n}\parallel_h^2  \leq \frac{1}{\left( 1-2\delta_t\alpha\right)} \sum\limits_{i=1}^{n} \parallel e^{i}\parallel_h^2 +  \frac{\delta_t}{\lambda\left( 1-2\delta_t\alpha\right)}\sum\limits_{i=0}^{n-1} \parallel \mathfrak{e}^{i}\parallel_h^2 + \frac{c_1T}{\left( 1-2\delta_t\alpha\right)}( \delta_t + h^2 )^2.
\end{equation}
An application of Gronwall's inequality for sufficiently small $\delta_t$ yields
\begin{equation}\label{est_state6}
\parallel e^{n}\parallel_h^2  \leq   C \delta_t\sum\limits_{i=0}^{n-1} \parallel \mathfrak{e}^{i}\parallel_h^2 + C ( \delta_t + h^2 )^2.
\end{equation}
Taking now the inner product of \eqref{err_adjoint} and $\mathfrak{e}^{n}$ yields
\begin{equation}\label{est_adjoint1}
\left(-\frac{\mathfrak{e}^{n+1}-\mathfrak{e}^n}{\delta_t}, \mathfrak{e}^{n}\right)_h=\left(f'(y^{n})\Delta_h p^n -f'(Y^{n})\Delta_h P^n, \mathfrak{e}^{n} \right)_h -\epsilon^2 \vert \mathfrak{e}^{n+1}\vert_{2,h}^2 -\left(e^{n+1}, \mathfrak{e}^{n}\right)_h + \left(G^n, \mathfrak{e}^{n}\right)_h.
\end{equation}
Note that 
\begin{equation}\label{nonlin_1}
\left(f'(y^{n})\Delta_h p^n -f'(Y^{n})\Delta_h P^n, \mathfrak{e}^{n} \right)_h \leq \left( \vert \left(f'(y^{n}) +f'(Y^{n})\right) \vert \,\vert \Delta_h \mathfrak{e}^n \vert, \mathfrak{e}^{n} \right)_h \leq 2M \| \mathfrak{e}^n\|_h \vert \mathfrak{e}^{n+1}\vert_{2,h}.
\end{equation}
Using \eqref{nonlin_1} and $\left(-\frac{\mathfrak{e}^{n+1}-\mathfrak{e}^n}{\delta_t}, \mathfrak{e}^{n}\right)_h \geq \frac{\parallel \mathfrak{e}^{n}\parallel_h^2 - \parallel\mathfrak{ e}^{n+1}\parallel_h^2}{2\delta_t}$ on \eqref{est_adjoint1} we obtain 
\begin{equation}\label{est_adjoint2}
\frac{\parallel \mathfrak{e}^{n}\parallel_h^2 - \parallel\mathfrak{ e}^{n+1}\parallel_h^2}{2\delta_t} \leq 2M \parallel \mathfrak{e}^n\parallel_h \vert \mathfrak{e}^{n+1}\vert_{2,h} -\epsilon^2 \vert \mathfrak{e}^{n+1}\vert_{2,h}^2 -\left(e^{n+1}, \mathfrak{e}^{n}\right)_h + \left(G^n, \mathfrak{e}^{n}\right)_h.
\end{equation}
Using Young's inequality and \eqref{truncation_err} we get 
\begin{equation}\label{est_adjoint3}
\frac{\parallel \mathfrak{e}^{n}\parallel_h^2 - \parallel\mathfrak{ e}^{n+1}\parallel_h^2}{2\delta_t} \leq \underbrace{\left(\frac{M^2}{\epsilon^2} +1\right)}_{\beta} \parallel \mathfrak{e}^n\parallel_h^2 + \frac{1}{2}\parallel {e}^{n+1}\parallel_h^2 + \frac{c_2}{2} ( \delta_t + h^2 )^2.
\end{equation}
Thus we obtain
\begin{equation}\label{est_adjoint4}
(1-2\delta_t\beta)\parallel \mathfrak{e}^{n}\parallel_h^2 \leq  \parallel\mathfrak{ e}^{n+1}\parallel_h^2 +  \delta_t\parallel {e}^{n+1}\parallel_h^2 + c_2\delta_t ( \delta_t + h^2 )^2.
\end{equation}
For $\delta_t<\frac{1}{2\beta}$ in \eqref{est_adjoint4} and taking the sum over $n$ we have
\begin{equation}\label{est_adjoint5}
\parallel \mathfrak{e}^{n}\parallel_h^2 \leq \frac{1}{(1-2\delta_t\beta)}\sum\limits_{i=0}^{n-1} \parallel\mathfrak{ e}^{i}\parallel_h^2 +  \frac{\delta_t}{(1-2\delta_t\beta)} \sum\limits_{i=1}^{n}\parallel {e}^{i}\parallel_h^2 + \frac{c_2T}{(1-2\delta_t\beta)} ( \delta_t + h^2 )^2.
\end{equation}
Applying Gronwall's inequality on \eqref{est_adjoint5} for sufficiently small $\delta_t$ we obtain
\begin{equation}\label{est_adjoint6}
\parallel \mathfrak{e}^{n}\parallel_h^2  \leq   C \delta_t\sum\limits_{i=1}^{n} \parallel e^{i}\parallel_h^2 + C ( \delta_t + h^2 )^2.
\end{equation}
By adding \eqref{est_state6} and \eqref{est_adjoint6} we have 
\begin{equation}\label{est_1}
\parallel e^{n}\parallel_h^2 + \parallel \mathfrak{e}^{n}\parallel_h^2  \leq   C \delta_t\sum\limits_{i=1}^{n} \parallel e^{i}\parallel_h^2 + C \delta_t\sum\limits_{i=0}^{n-1} \parallel \mathfrak{e}^{i}\parallel_h^2 + C ( \delta_t + h^2 )^2.
\end{equation}
Again using Gronwall's inequality on \eqref{est_1} for sufficiently small $\delta_t$ we get
\begin{equation}\label{est_2}
\parallel e^{n}\parallel_h + \parallel \mathfrak{e}^{n}\parallel_h  \leq    C ( \delta_t + h^2 ).
\end{equation}
Next we take the inner product of \eqref{err_state} and $\partial_t e^n:=\frac{e^{n+1}-e^n}{\delta_t}$, which yields
\begin{equation}\label{est_state11}
\parallel \partial_t e^n\parallel^2_h=\left(\Delta_h \left(f(y^{n+1})-f(Y^{n+1})\right), \partial_t e^n\right)_h-\epsilon^2 \left(\Delta_h e^{n+1}, \partial_t( \Delta_he^n)\right)_h + \frac{1}{\lambda}\left(\mathfrak{e}^{n}, \partial_te^{n}\right)_h + \left(F^n, \partial_te^{n}\right)_h.
\end{equation}
Using the Lipschitz condition of $f$ and Cauchy-Schwarz inequality on \eqref{est_state11} we get
\begin{multline}\label{est_state12}
\parallel \partial_t e^n\parallel^2_h + \epsilon^2 \left(\Delta_h e^{n+1}, \partial_t( \Delta_he^n)\right)_h \leq M \parallel e^{n+1}\parallel_h \parallel \partial_t( \Delta_he^n)\parallel_h \\
+ \frac{1}{\lambda}\parallel\mathfrak{e}^{n}\parallel_h \parallel\partial_te^{n}\parallel_h + \parallel F^n\parallel_h \parallel\partial_te^{n}\parallel_h.
\end{multline}
Using Young's inequality and $\left(\partial_t( \Delta_he^n), \Delta_h e^{n+1}\right)_h \geq \frac{\vert e^{n+1}\vert_{2,h}^2 - \vert e^{n}\vert_{2,h}^2}{2\delta_t}$ on \eqref{est_state12} we have the following
\begin{equation}\label{est_state13}
\epsilon^2 \frac{\vert e^{n+1}\vert_{2,h}^2 - \vert e^{n}\vert_{2,h}^2}{2\delta_t} \leq \frac{M^2}{\epsilon^2} \parallel e^{n+1}\parallel_h^2 + \frac{\epsilon^2}{4} \parallel \partial_t( \Delta_he^n)\parallel_h^2 + \frac{1}{2\lambda^2}\parallel\mathfrak{e}^{n}\parallel_h^2 + \frac{1}{2}\parallel F^n\parallel_h^2.
\end{equation}
Using \eqref{est_2} and \eqref{truncation_err} on \eqref{est_state13} we have
\begin{equation}\label{est_state14}
\epsilon^2 \frac{\vert e^{n+1}\vert_{2,h}^2 - \vert e^{n}\vert_{2,h}^2}{2\delta_t} \leq  \frac{\epsilon^2}{4} \frac{\vert e^{n+1}\vert_{2,h}^2 + \vert e^{n}\vert_{2,h}^2}{\delta_t} +C(\delta_t + h^2)^2.
\end{equation}
Thus the equation \eqref{est_state14} implies
\begin{equation}\label{est_state15}
\frac{\epsilon^2}{2} \vert e^{n+1}\vert_{2,h}^2  \leq  \frac{3\epsilon^2}{2}  \vert e^{n}\vert_{2,h}^2 +2\delta_tC(\delta_t + h^2)^2.
\end{equation}
Taking the sum over $n$ on \eqref{est_state15} we find that
\begin{equation}\label{est_state16}
\frac{\epsilon^2}{2} \vert e^{n}\vert_{2,h}^2  \leq  \frac{3\epsilon^2}{2}  \sum\limits_{i=1}^{n}\vert e^{i}\vert_{2,h}^2 +2TC(\delta_t + h^2)^2.
\end{equation}
Applying Gronwall's inequality on \eqref{est_state16} we get
\begin{equation}\label{est_3}
 \vert e^{n}\vert_{2,h}  \leq  C(\delta_t + h^2).
\end{equation}
Also we have 
$\vert e^n\vert_{1,h}^2=-(\Delta_h e^n, e^n)\leq \vert e^n\vert_{2,h}\parallel e^n\parallel_h$. Then using \eqref{est_3} and \eqref{est_2} we obtain
\begin{equation}\label{est_4}
 \vert e^{n}\vert_{1,h}  \leq  C(\delta_t + h^2).
\end{equation}
Now we take the inner product of \eqref{err_adjoint} and $-\partial_t \mathfrak{e}^n$, that yields
\begin{multline}\label{est_adjoint11}
\parallel\partial_t \mathfrak{e}^n\parallel_h^2=\left(f'(y^{n})\Delta_h p^n -f'(Y^{n})\Delta_h P^n, -\partial_t \mathfrak{e}^n \right)_h -\epsilon^2 \left(\Delta_h \mathfrak{e}^{n+1}, -\partial_t( \Delta_h \mathfrak{e}^n)\right)_h\\
 -\left(e^{n+1}, -\partial_t \mathfrak{e}^n\right)_h + \left(G^n, -\partial_t \mathfrak{e}^n\right)_h.
\end{multline}
Following \eqref{nonlin_1} and using Cauchy-Schwarz inequality on \eqref{est_adjoint11} we have 
\begin{equation}\label{est_adjoint12}
\parallel\partial_t \mathfrak{e}^n\parallel_h^2 + \epsilon^2 \left(\Delta_h \mathfrak{e}^{n+1}, -\partial_t( \Delta_h \mathfrak{e}^n)\right)_h \leq 2M \vert \mathfrak{e}^n \vert_{2, h} \parallel\partial_t \mathfrak{e}^n\parallel_h + \parallel e^{n+1}\parallel_h \parallel\partial_t \mathfrak{e}^n\parallel_h + \parallel G^n\parallel_h \parallel\partial_t \mathfrak{e}^n\parallel_h.
\end{equation}
Using Young's inequality on \eqref{est_adjoint12} we get
\begin{equation}\label{est_adjoint13}
 \epsilon^2 \frac{\vert \mathfrak{e}^{n}\vert_{2,h}^2 - \vert \mathfrak{e}^{n+1}\vert_{2,h}^2}{2\delta_t} \leq 4M^2 \vert \mathfrak{e}^n \vert_{2, h}^2 + \frac{1}{2}\parallel e^{n+1}\parallel_h^2 + \parallel G^n\parallel_h^2.
\end{equation}
Using \eqref{truncation_err} and \eqref{est_2} on \eqref{est_adjoint13} we have
\begin{equation}\label{est_adjoint14}
 (\epsilon^2-8M^2\delta_t) \vert \mathfrak{e}^{n}\vert_{2,h}^2 \leq \epsilon^2 \vert \mathfrak{e}^{n+1} \vert_{2, h}^2 + 2\delta_t C(\delta_t + h^2)^2.
\end{equation}
Letting $\delta_t<\frac{\epsilon^2}{8M^2}$ in \eqref{est_adjoint14} and summing over $n$ yields
\begin{equation}\label{est_adjoint15}
\vert \mathfrak{e}^{n}\vert_{2,h}^2 \leq \frac{\epsilon^2}{ (\epsilon^2-8M^2\delta_t) } \sum\limits_{i=0}^{n-1}\vert \mathfrak{e}^{i} \vert_{2, h}^2 + 2T C(\delta_t + h^2)^2.
\end{equation}
Employing Gronwall's inequality in \eqref{est_adjoint15} we obtain
\begin{equation}\label{est_5}
\vert \mathfrak{e}^{n}\vert_{2,h} \leq C(\delta_t + h^2),
\end{equation}
and from \eqref{est_5}  and \eqref{est_2} we have
\begin{equation}\label{est_6}
\vert \mathfrak{e}^{n}\vert_{1,h} \leq C(\delta_t + h^2).
\end{equation}
Then using \eqref{est_2}, \eqref{est_4} and \eqref{est_6} in discrete Sobolev's embedding theorem, we derive that 
\begin{equation*}
\parallel {e}^{n}\parallel_{\infty,h} + \parallel\mathfrak{e}^{n}\parallel_{\infty,h} \leq C(\delta_t + h^2).
\end{equation*}
Hence the theorem.
\end{proof}
\begin{theorem}[Convergence of the scheme $\textbf{S}2$]\label{thm2_ocp}
Let $y(x,t)$ and $p(x,t)$ be sufficiently smooth functions. For $\delta_t$ sufficiently small, the finite difference scheme \eqref{discrete_state2}-\eqref{discrete_adjoint} is first order in time and second order in space convergent, i.e., 
$$  \max\limits_n \left\{\parallel  y^n-Y^n \parallel_{\infty,h} + \parallel p^n-P^n \parallel_{\infty, h} \right\}  \leq C ( \delta_t + h^2 ).$$
\end{theorem}
\begin{proof}
We provide an estimate that is consistent with the linear state approximation scheme \eqref{discrete_state2}, and for an estimate that is consistent with discrete adjoint, we turn to Theorem \ref{thm1_ocp}. Using Taylor expansion we can see that the exact solution $y_i^n=y(x_i, t_n)$ and $p_i^n=p(x_i, t_n)$ satisfy the following equations
\begin{equation}\label{taylor_state2}
\frac{y_i^{n+1}-y_i^n}{\delta_t}=\Delta_h (y_i^{n})^2y_i^{n+1} -\Delta_h y_i^n -\epsilon^2 \Delta_h^2y_i^{n+1} + \frac{1}{\lambda}p_i^{n} +\widetilde{ F}_i^n,
\end{equation}
where $\widetilde{F}_i^n$ denotes the truncation error, which satisfies the following for some positive constant $c_3$,
 \begin{equation}\label{truncation_err2}
\max\limits_{i,n} \vert \widetilde{F}_i^n\vert \leq c_3 ( \delta_t + h^2 ). 
 \end{equation}
Let us define the error $e_i^n=y_i^n-Y_i^n$ and $\mathfrak{e}_i^n=p_i^n-P_i^n$. Then taking the difference between \eqref{taylor_state2} and \eqref{discrete_state2} yields
\begin{equation}\label{err_state2}
\frac{e_i^{n+1}-e_i^n}{\delta_t}=\Delta_h \left((y_i^{n})^2y_i^{n+1}-(Y_i^{n})^2 Y_i^{n+1}\right)-\Delta_h e_i^n-\epsilon^2 \Delta_h^2e_i^{n+1} + \frac{1}{\lambda}\mathfrak{e}_i^{n} + \widetilde{ F}_i^n.
\end{equation} 
Taking the inner product of \eqref{err_state2} and $e^{n+1}$, we get
\begin{multline}\label{err_state21}
\left(\frac{e^{n+1}-e^n}{\delta_t}, e^{n+1}\right)_h=\left( \Delta_h \left((y^{n})^2y^{n+1}-(Y^{n})^2 Y^{n+1}\right), e^{n+1}\right) -\left(\Delta_h e^n, e^{n+1}\right)_h
\\ -\epsilon^2 \vert e^{n+1}\vert_{2,h}^2 + \frac{1}{\lambda}\left(\mathfrak{e}^{n}, e^{n+1}\right)_h + \left(\widetilde{F}^n, e^{n+1}\right)_h.
\end{multline}
Observe that
\[
\left( \Delta_h \left((y^{n})^2y^{n+1}-(Y^{n})^2 Y^{n+1}\right), e^{n+1}\right) \leq \left(  \vert \widetilde{f}'(y^n) + \widetilde{f}'(Y^n) \vert \vert e^{n+1}\vert, \Delta_h e^{n+1}\right) \leq 2\widetilde{M} \parallel e^{n+1}\parallel_h \vert e^{n+1}\vert_{2,h}. 
\]
Using the above estimate, Young's inequality and \eqref{truncation_err2} on \eqref{err_state21} we obtain 
\begin{multline}\label{err_state22}
\frac{\parallel e^{n+1}\parallel_h^2 - \parallel e^{n}\parallel_h^2}{2\delta_t} +\frac{\epsilon^2}{2} \vert e^{n+1}\vert_{2,h}^2 \leq \underbrace{\left(\frac{4\widetilde{M}^2}{\epsilon^2} 
+ \frac{1}{2\epsilon^2} + \frac{1}{2\lambda} + \frac{1}{2} \right)}_{\widetilde{\alpha}} \parallel e^{n+1}\parallel_h^2 \\ + \frac{\epsilon^2}{2} \vert e^{n}\vert_{2,h}^2 + \frac{1}{2\lambda}\parallel \mathfrak{e}^{n}\parallel_h^2 + \frac{c_3}{2} ( \delta_t + h^2 )^2.
\end{multline}
From \eqref{err_state22} we have 
\begin{equation}\label{err_state23}
\left( 1-2\delta_t\widetilde{\alpha}\right) \parallel e^{n+1}\parallel_h^2 + \epsilon^2 \delta_t\left(\vert e^{n+1}\vert_{2,h}^2 -\vert e^{n}\vert_{2,h}^2 \right)\leq \parallel e^{n}\parallel_h^2 + \frac{\delta_t}{\lambda}\parallel \mathfrak{e}^{n}\parallel_h^2 + c_3\delta_t ( \delta_t + h^2 )^2.
\end{equation}
For $\delta_t<\frac{1}{2\widetilde{\alpha}}$ in \eqref{err_state23} and taking the sum over $n$ we get
\begin{equation}\label{err_state24}
\parallel e^{n}\parallel_h^2  \leq \frac{1}{\left( 1-2\delta_t\widetilde{\alpha}\right)} \sum\limits_{i=1}^{n} \parallel e^{i}\parallel_h^2 +  \frac{\delta_t}{\lambda\left( 1-2\delta_t\widetilde{\alpha}\right)}\sum\limits_{i=0}^{n-1} \parallel \mathfrak{e}^{i}\parallel_h^2 + \frac{c_3T}{\left( 1-2\delta_t\widetilde{\alpha}\right)}( \delta_t + h^2 )^2.
\end{equation}
Applying Gronwall's inequality on \eqref{err_state24} for sufficiently small $\delta_t$ yields
\begin{equation}\label{err_state25}
\parallel e^{n}\parallel_h^2  \leq   C \delta_t\sum\limits_{i=0}^{n-1} \parallel \mathfrak{e}^{i}\parallel_h^2 + C ( \delta_t + h^2 )^2.
\end{equation}
Adding \eqref{est_adjoint6} (the error estimates in Theorem \ref{thm1_ocp} for adjoint equation) and \eqref{err_state25}, and employing Gronwall's inequality for sufficiently small $\delta_t$ we have 
\begin{equation}\label{err_state26}
\parallel e^{n}\parallel_h + \parallel \mathfrak{e}^{n}\parallel_h  \leq    C ( \delta_t + h^2 ).
\end{equation}
Next we take the inner product of \eqref{err_state2} and $\partial_t e^n:=\frac{e^{n+1}-e^n}{\delta_t}$, that yields
\begin{equation}\label{err_state27}
\begin{aligned}
\parallel \partial_t e^n\parallel^2_h=&\left( \Delta_h \left((y^{n})^2y^{n+1}-(Y^{n})^2 Y^{n+1}\right), \partial_t e^n\right) -\left(\Delta_h e^n, \partial_t e^n\right)_h\\
&-\epsilon^2 \left(\Delta_h e^{n+1}, \partial_t( \Delta_he^n)\right)_h + \frac{1}{\lambda}\left(\mathfrak{e}^{n}, \partial_te^{n}\right)_h + \left(\widetilde{F}^n, \partial_te^{n}\right)_h,\\
& \leq 2\widetilde{M} \parallel e^{n+1}\parallel_h \parallel  \partial_t( \Delta_he^n)\parallel_h + \vert e^{n}\vert_{2,h}^2 + \frac{1}{4} \parallel \partial_t e^n\parallel^2_h \\
& -\epsilon^2 \left(\Delta_h e^{n+1}, \partial_t( \Delta_he^n)\right)_h + \frac{1}{2\lambda^2}\parallel\mathfrak{e}^{n}\parallel_h^2+\frac{1}{2} \parallel \partial_t e^n\parallel^2_h + \parallel \widetilde{F}^n\parallel_h^2 + \frac{1}{4} \parallel \partial_t e^n\parallel^2_h.
\end{aligned}
\end{equation}
Using \eqref{truncation_err2} and \eqref{est_2} on \eqref{err_state27} we get
\begin{equation}\label{err_state28}
\epsilon^2 \frac{\vert e^{n+1}\vert_{2,h}^2 - \vert e^{n}\vert_{2,h}^2}{2\delta_t} \leq  \frac{\epsilon^2}{4} \frac{\vert e^{n+1}\vert_{2,h}^2 + \vert e^{n}\vert_{2,h}^2}{\delta_t} +  \vert e^{n}\vert_{2,h}^2 +C(\delta_t + h^2)^2.
\end{equation}
From \eqref{err_state28} we have 
\begin{equation}\label{err_state29}
\frac{\epsilon^2}{2} \vert e^{n+1}\vert_{2,h}^2  \leq  \left(\frac{3\epsilon^2}{2} + 2\delta_t \right)  \vert e^{n}\vert_{2,h}^2 +2\delta_tC(\delta_t + h^2)^2.
\end{equation}
Taking the sum over $n$ on \eqref{err_state29} we get
\begin{equation}\label{err_state290}
\frac{\epsilon^2}{2} \vert e^{n}\vert_{2,h}^2  \leq  \left(\frac{3\epsilon^2}{2} + 2\delta_t \right) \sum\limits_{i=1}^{n}\vert e^{i}\vert_{2,h}^2 +2TC(\delta_t + h^2)^2.
\end{equation}
Applying Gronwall's inequality on \eqref{err_state290} we have 
\begin{equation}\label{err_state291}
 \vert e^{n}\vert_{2,h}  \leq  C(\delta_t + h^2).
\end{equation}
Using \eqref{err_state26} and \eqref{err_state291} we obtain
\begin{equation}\label{err_state292}
 \vert e^{n}\vert_{1,h}  \leq  C(\delta_t + h^2).
\end{equation}
Now using \eqref{err_state26}, \eqref{err_state292} and \eqref{est_6} in discrete Sobolev's embedding theorem we get our result.
\end{proof}
\begin{theorem}[Convergence of the scheme $\textbf{S}3$]\label{thm3_ocp}
Let $y(x,t)$ and $p(x,t)$ be sufficiently smooth functions. For $\delta_t$ sufficiently small, the finite difference scheme \eqref{discrete_state3}-\eqref{discrete_adjoint} is first order in time and second order in space convergent, i.e., 
$$  \max\limits_n \left\{\parallel  y^n-Y^n \parallel_{\infty,h} + \parallel p^n-P^n \parallel_{\infty, h} \right\}  \leq C ( \delta_t + h^2 ).$$
\end{theorem}
\begin{proof}
We provide an estimate that is consistent with the linear state approximation scheme \eqref{discrete_state2}, and for an estimate that is consistent with discrete adjoint, we turn to Theorem \ref{thm1_ocp}. Using Taylor expansion we can see that the exact solution $y_i^n=y(x_i, t_n)$ and $p_i^n=p(x_i, t_n)$ satisfy the following equation
\begin{equation}\label{taylor_state3}
\frac{y_i^{n+1}-y_i^n}{\delta_t}=\Delta_h (y_i^{n})^3  -3\Delta_h y_i^n + 2\Delta_h y_i^{n+1}-\epsilon^2 \Delta_h^2y_i^{n+1} + \frac{1}{\lambda}p_i^{n} + \widehat{F}_i^n,
\end{equation} 
where $\widehat{F}_i^n$ denotes the truncation error, which satisfies the following for some positive constant $c_4$
 \begin{equation}\label{truncation_err3}
\max\limits_{i,n} \vert \widehat{F}_i^n\vert \leq c_4 ( \delta_t + h^2 ). 
 \end{equation}
Let us define the error $e_i^n=y_i^n-Y_i^n$ and $\mathfrak{e}_i^n=p_i^n-P_i^n$. Then taking the difference between \eqref{taylor_state3} and \eqref{discrete_state3} yields
\begin{equation}\label{err_state31}
\frac{e_i^{n+1}-e_i^n}{\delta_t}=\Delta_h \left((y_i^{n})^3-(Y_i^{n})^3 \right)-3\Delta_h e_i^n + 2\Delta_h e_i^{n+1}-\epsilon^2 \Delta_h^2e_i^{n+1} + \frac{1}{\lambda}\mathfrak{e}_i^{n} + \widehat{ F}_i^n.
\end{equation} 
Taking the inner product of \eqref{err_state31} and $e^{n+1}$ yields
\begin{equation}\label{err_state32}
\begin{aligned}
\left(\frac{e^{n+1}-e^n}{\delta_t}, e^{n+1}\right)_h=&\left( \Delta_h \left((y_i^{n})^3-(Y_i^{n})^3\right), e^{n+1}\right) -3\left(\Delta_h e^n, e^{n+1}\right)_h 
\\ & + 2\left(\Delta_h e^{n+1}, e^{n+1}\right)_h -\epsilon^2 \vert e^{n+1}\vert_{2,h}^2 + \frac{1}{\lambda}\left(\mathfrak{e}^{n}, e^{n+1}\right)_h + \left(\widehat{F}^n, e^{n+1}\right)_h.
\end{aligned}
\end{equation}
Using the differentiability of $f$ and Cauchy-Schwarz on \eqref{err_state32} we have 
\begin{equation}\label{err_state33}
\begin{aligned}
\frac{\parallel e^{n+1}\parallel_h^2 - \parallel e^{n}\parallel_h^2}{2\delta_t} &\leq M \parallel e^{n}\parallel_h \vert e^{n+1}\vert_{2,h}+ 3\parallel e^{n+1}\parallel_h \vert e^{n}\vert_{2,h} + 
2\parallel e^{n+1}\parallel_h \vert e^{n+1}\vert_{2,h}\\
&-{\epsilon^2} \vert e^{n+1}\vert_{2,h}^2 + \frac{1}{\lambda}\parallel \mathfrak{e}^{n}\parallel_h \parallel e^{n+1}\parallel_h + \parallel \widehat{F}^{n}\parallel_h \parallel e^{n+1}\parallel_h\\
& \leq \frac{M^2}{\epsilon^2} \parallel e^{n}\parallel_h^2 + \frac{\epsilon^2}{4}\vert e^{n+1}\vert_{2,h} + \frac{9}{2\epsilon^2} \parallel e^{n+1}\parallel_h^2 + \frac{\epsilon^2}{2}\vert e^{n}\vert_{2,h} + \frac{4}{\epsilon^2} \parallel e^{n+1}\parallel_h^2 + \frac{\epsilon^2}{4}\vert e^{n+1}\vert_{2,h} \\
& -{\epsilon^2} \vert e^{n+1}\vert_{2,h}^2 + \frac{1}{2\lambda^2}\parallel \mathfrak{e}^{n}\parallel_h^2 + \frac{1}{2} \parallel e^{n+1}\parallel_h^2 + \frac{1}{2} \parallel \widehat{F}^{n}\parallel_h^2 + \frac{1}{2} \parallel e^{n+1}\parallel_h^2,\\
& \leq \frac{M^2}{\epsilon^2} \parallel e^{n}\parallel_h^2 + \underbrace{\left(\frac{17}{2\epsilon^2} + 1 \right)}_{\alpha_0} \parallel e^{n+1}\parallel_h^2 -\frac{\epsilon^2}{2} \vert e^{n+1}\vert_{2,h}^2 + \frac{\epsilon^2}{2} \vert e^{n}\vert_{2,h}^2 \\
&\quad \quad \quad + \frac{1}{2\lambda^2}\parallel \mathfrak{e}^{n}\parallel_h^2 + \frac{c_4}{2} (\delta_t + h^2)^2,
\end{aligned}
\end{equation}
where on the second inequality we use Young's inequality and in the last inequality we use \eqref{truncation_err3}. From \eqref{err_state33} we get
\begin{equation}\label{err_state34}
\left( 1-2\delta_t\alpha_{0}\right) \parallel e^{n+1}\parallel_h^2 + \epsilon^2 \delta_t\left(\vert e^{n+1}\vert_{2,h}^2 -\vert e^{n}\vert_{2,h}^2 \right)\leq \left( 1 +\frac{2\delta_t M^2}{\epsilon^2} \right) \parallel e^{n}\parallel_h^2 + \frac{\delta_t}{\lambda^2}\parallel \mathfrak{e}^{n}\parallel_h^2 + c_4\delta_t ( \delta_t + h^2 )^2.
\end{equation}
For $\delta_t<\frac{1}{2{\alpha_0}}$ we take the sum over $n$ to obtain
\begin{equation}\label{err_state35}
\parallel e^{n}\parallel_h^2  \leq \frac{\left( 1 +\frac{2\delta_t M^2}{\epsilon^2} \right)}{\left( 1-2\delta_t{\alpha_0}\right)} \sum\limits_{i=1}^{n} \parallel e^{i}\parallel_h^2 +  \frac{\delta_t}{\lambda^2\left( 1-2\delta_t{\alpha_0}\right)}\sum\limits_{i=0}^{n-1} \parallel \mathfrak{e}^{i}\parallel_h^2 + \frac{c_4T}{\left( 1-2\delta_t{\alpha_0}\right)}( \delta_t + h^2 )^2.
\end{equation}
Applying Gronwall's inequality on \eqref{err_state35} for sufficiently small $\delta_t$, we get
\begin{equation}\label{err_state36}
\parallel e^{n}\parallel_h^2  \leq   C \delta_t\sum\limits_{i=0}^{n-1} \parallel \mathfrak{e}^{i}\parallel_h^2 + C ( \delta_t + h^2 )^2.
\end{equation}
Adding \eqref{est_adjoint6} (error estimates in Theorem \ref{thm1_ocp} for adjoint equation) and \eqref{err_state36} , and employing Gronwall's inequality for sufficiently small $\delta_t$ we have  
\begin{equation}\label{err_state37}
\parallel e^{n}\parallel_h + \parallel \mathfrak{e}^{n}\parallel_h  \leq    C ( \delta_t + h^2 ).
\end{equation}
Next we take the inner product of \eqref{err_state31} and $\partial_t e^n:=\frac{e^{n+1}-e^n}{\delta_t}$, that produces
\begin{equation}\label{err_state38}
\begin{aligned}
\parallel \partial_t e^n\parallel^2_h=&\left( \Delta_h \left((y_i^{n})^3-(Y_i^{n})^3\right), \partial_t e^n\right) -3\left(\Delta_h e^n, \partial_t e^n\right)_h + 2\left(\Delta_h e^{n+1}, \partial_t e^n\right)_h\\
&-\epsilon^2 \left(\Delta_h e^{n+1}, \partial_t( \Delta_he^n)\right)_h + \frac{1}{\lambda}\left(\mathfrak{e}^{n}, \partial_te^{n}\right)_h + \left(\widehat{F}^n, \partial_te^{n}\right)_h,\\
& \leq {M} \parallel e^{n}\parallel_h \parallel  \partial_t( \Delta_he^n)\parallel_h + 9\vert e^{n}\vert_{2,h}^2 + \frac{1}{4} \parallel \partial_t e^n\parallel^2_h + 4\vert e^{n+1}\vert_{2,h}^2 + \frac{1}{4} \parallel \partial_t e^n\parallel^2_h \\
& -\epsilon^2 \left(\Delta_h e^{n+1}, \partial_t( \Delta_he^n)\right)_h + \frac{1}{\lambda^2}\parallel\mathfrak{e}^{n}\parallel_h^2+\frac{1}{4} \parallel \partial_t e^n\parallel^2_h + \parallel \widehat{F}^n\parallel_h^2 + \frac{1}{4} \parallel \partial_t e^n\parallel^2_h,
\end{aligned}
\end{equation}
where on the second inequality we use differentiability of $f$ and Young's inequality. Using \eqref{err_state37}, \eqref{truncation_err3} and \eqref{est_2} on \eqref{err_state38} we have 
\begin{equation}\label{err_state39}
\epsilon^2 \left(\Delta_h e^{n+1}, \partial_t( \Delta_he^n)\right)_h \leq \frac{\epsilon^2}{4} \parallel  \partial_t( \Delta_he^n)\parallel_h^2 + 9\vert e^{n}\vert_{2,h}^2 + 4\vert e^{n+1}\vert_{2,h}^2 + C(\delta_t + h^2)^2.
\end{equation}
From \eqref{err_state39} we get
\begin{equation}\label{err_state40}
\epsilon^2 \frac{\vert e^{n+1}\vert_{2,h}^2 - \vert e^{n}\vert_{2,h}^2}{2\delta_t} \leq  \frac{\epsilon^2}{4} \frac{\vert e^{n+1}\vert_{2,h}^2 + \vert e^{n}\vert_{2,h}^2}{\delta_t} +  9\vert e^{n}\vert_{2,h}^2 +4\vert e^{n+1}\vert_{2,h}^2 +C(\delta_t + h^2)^2.
\end{equation}
This further implies
\begin{equation}\label{err_state41}
\left(\frac{\epsilon^2}{2}-4\delta_t\right) \vert e^{n+1}\vert_{2,h}^2  \leq  \left(\frac{3\epsilon^2}{2} + 9\delta_t \right)  \vert e^{n}\vert_{2,h}^2 +2\delta_tC(\delta_t + h^2)^2.
\end{equation}
For $\delta_t< \frac{\epsilon^2}{8}$ we now take the sum over $n$ in \eqref{err_state41} to obtain
\begin{equation}\label{err_state42}
\left(\frac{\epsilon^2}{2}-4\delta_t\right) \vert e^{n}\vert_{2,h}^2  \leq  \left(\frac{3\epsilon^2}{2} + 9\delta_t \right) \sum\limits_{i=1}^{n}\vert e^{i}\vert_{2,h}^2 +2TC(\delta_t + h^2)^2.
\end{equation}
Applying Gronwall's inequality on \eqref{err_state42} we have 
\begin{equation}\label{err_state43}
 \vert e^{n}\vert_{2,h}  \leq  C(\delta_t + h^2).
\end{equation}
Using \eqref{err_state37} and \eqref{err_state43} we obtain
\begin{equation}\label{err_state44}
 \vert e^{n}\vert_{1,h}  \leq  C(\delta_t + h^2).
\end{equation}
Then using \eqref{err_state37}, \eqref{err_state44} and \eqref{est_6} in discrete Sobolev's embedding theorem we get our theorem.
\end{proof}
\begin{remark}
To obtain the stability estimates of the proposed schemes in $\parallel \cdot \parallel_h$, one can take an inner product of discrete state equations \eqref{discrete_state}, \eqref{discrete_state2} and \eqref{discrete_state3} with $Y^{n+1}$ and inner product of discrete adjoint equation \eqref{discrete_adjoint} with $P^n$. By following the procedures of the aforementioned convergence proofs for all the schemes, one can obtain the estimate for $\delta_t \leq O(\epsilon^2)$, $\parallel Y^{n} \parallel_h + \parallel P^{n} \parallel_h \leq C \sum\limits_{i=1}^{n} \parallel\widehat{Y}^i\parallel_h$ for some generic constant $C$. However, in numerical experiments we observe that even for $\delta_t > O(\epsilon^2)$ the schemes are stable and accurate.
\end{remark}
\begin{remark}
Derivation of the proposed schemes in higher dimensions is straightforward in finite difference setting and proceeds in precisely the same manner as in the 1D case,  so we omit the specifics. We only show the numerics in higher dimensions.
\end{remark}
\begin{remark}
One can replace the term $f'(Y_i^{n})\Delta_h P_i^n$ in \eqref{discrete_adjoint} by $f'(Y_i^{n+1})\Delta_h P_i^n$ and get the same convergence results for the respective schemes.
\end{remark}
\section{Numerical Illustration}\label{section4}
We now investigate the solution behaviour of the OCP \eqref{OCP}-\eqref{CH_state} using the proposed discretization strategies. To solve the nonlinear scheme $\textbf{S}1$ we use Newton method at each time step with residual history of $1 e{-10}$. 
\subsection{Experiments in 1D}
First we run experiments in 1D and demonstrate the accuracy and stability of those proposed schemes in space and time. Moreover we run the schemes for different desired state.
\subsubsection{Accuracy test} 
Since the exact solution of the OCP \eqref{OCP}-\eqref{CH_state} is not known, a comparison between the solution of \eqref{OCP}-\eqref{CH_state} on a coarse mesh with that on a fine mesh is considered to check the accuracy. The error in the numerical solution is measured in $L^{\infty}(0,T; L^{\infty}(\Omega))$. We take the spatial computational domain $\Omega=(0, 1)$. The initial and desired state for the test are $y(x, 0)=\cos(2\pi x)$ and $\widehat{y}(x,t)=\cos(2\pi x)e^{-t}$ respectively.  First, we check the precision with respect to step size
for all the proposed schemes for a time window $[0,T=0.1]$. We discretize the spatial domain by taking a mesh size $h=1/256$. 
\begin{figure}[h!]
    \centering
    \subfloat{{\includegraphics[height=5cm,width=6.5cm]{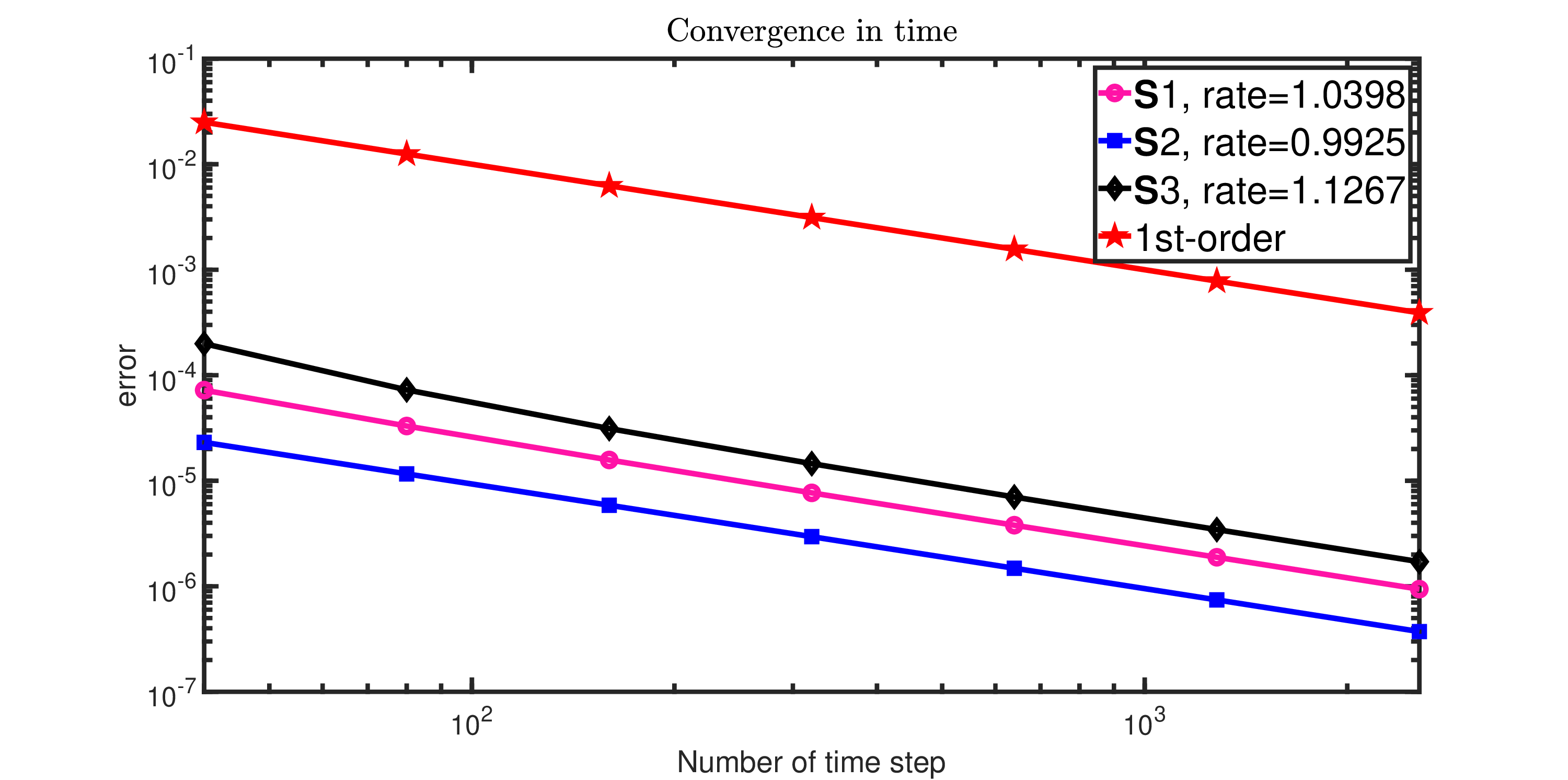} }}
     \subfloat{{\includegraphics[height=5cm,width=6.5cm]{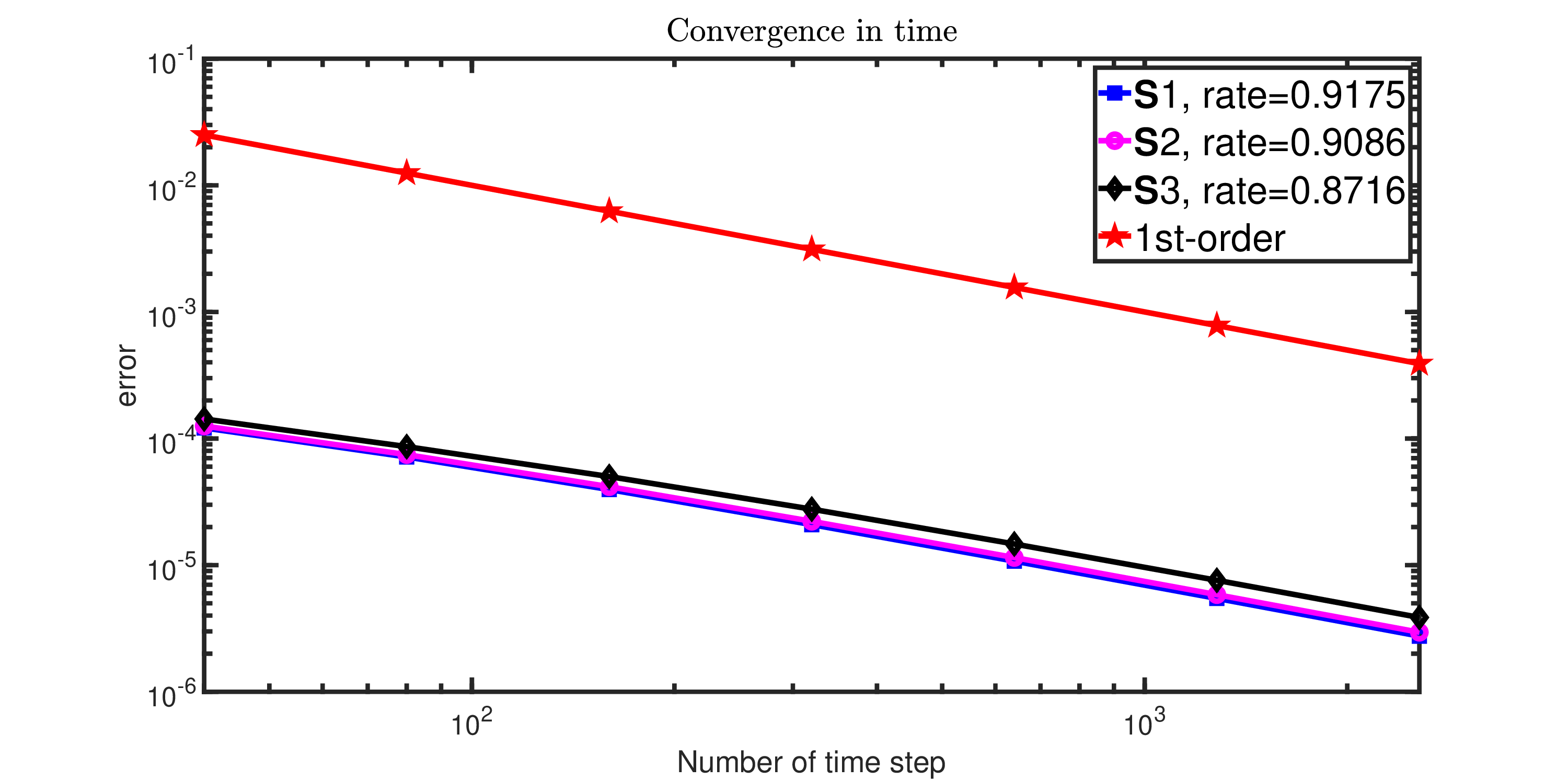} }}
    \caption{Convergence curves in time in $\log-\log$ scale with $\epsilon=0.05,  \lambda=0.1$ : on the left for the state equation, on the right for the adjoint equation}
    \label{accuracy_time}
\end{figure}
In Figure \ref{accuracy_time} we observe that the first order convergence rate in time for both state and adjoint equation is achieved for all the proposed schemes. Also note that schemes are stable and accurate for $\delta_t>O(\epsilon^2)$. To test the accuracy in mesh size we fix the final time at $T=0.01$ and $N_t=200$.
\begin{figure}[h!]
    \centering
    \subfloat{{\includegraphics[height=5cm,width=6.5cm]{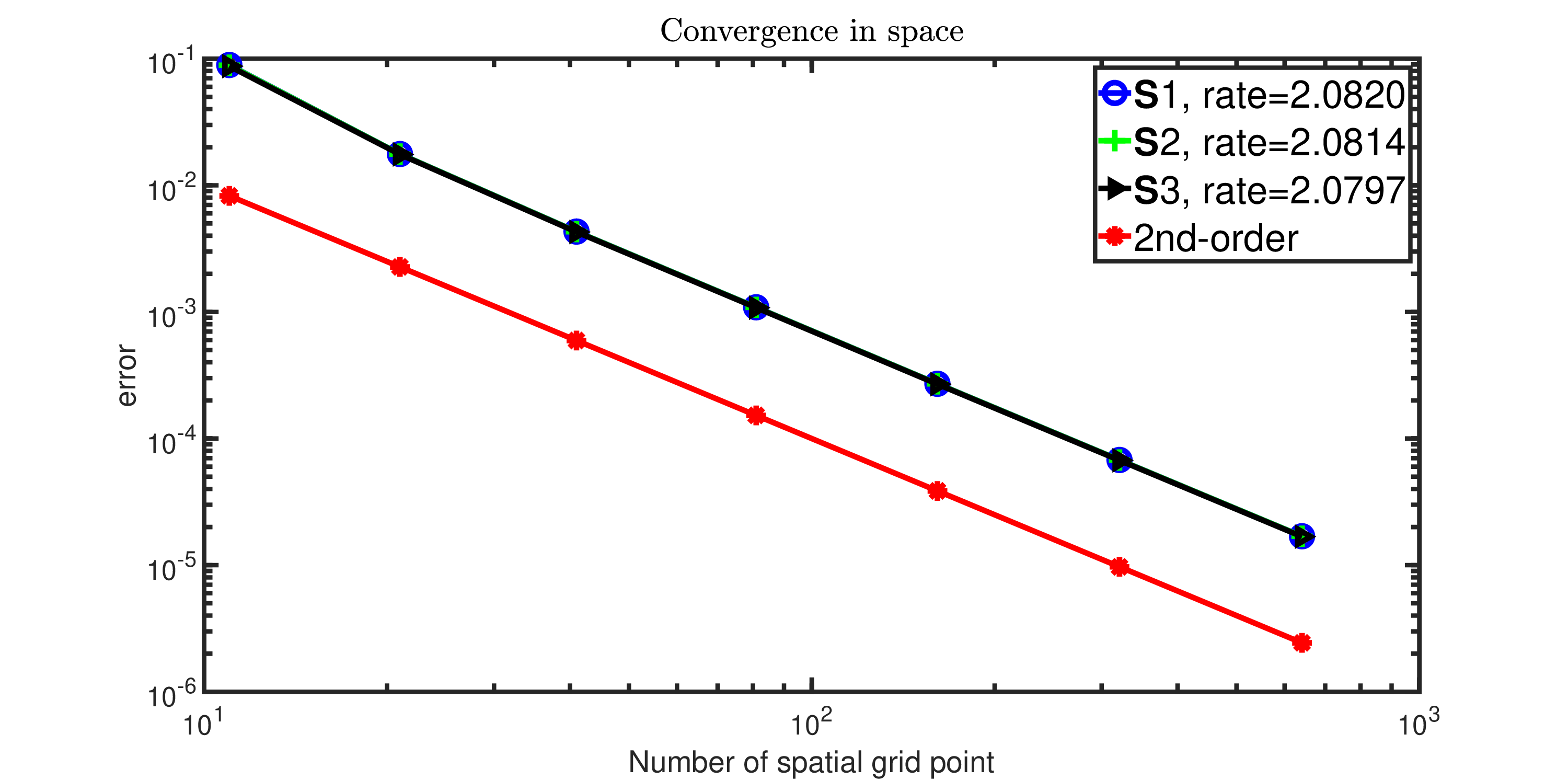} }}
     \subfloat{{\includegraphics[height=5cm,width=6.5cm]{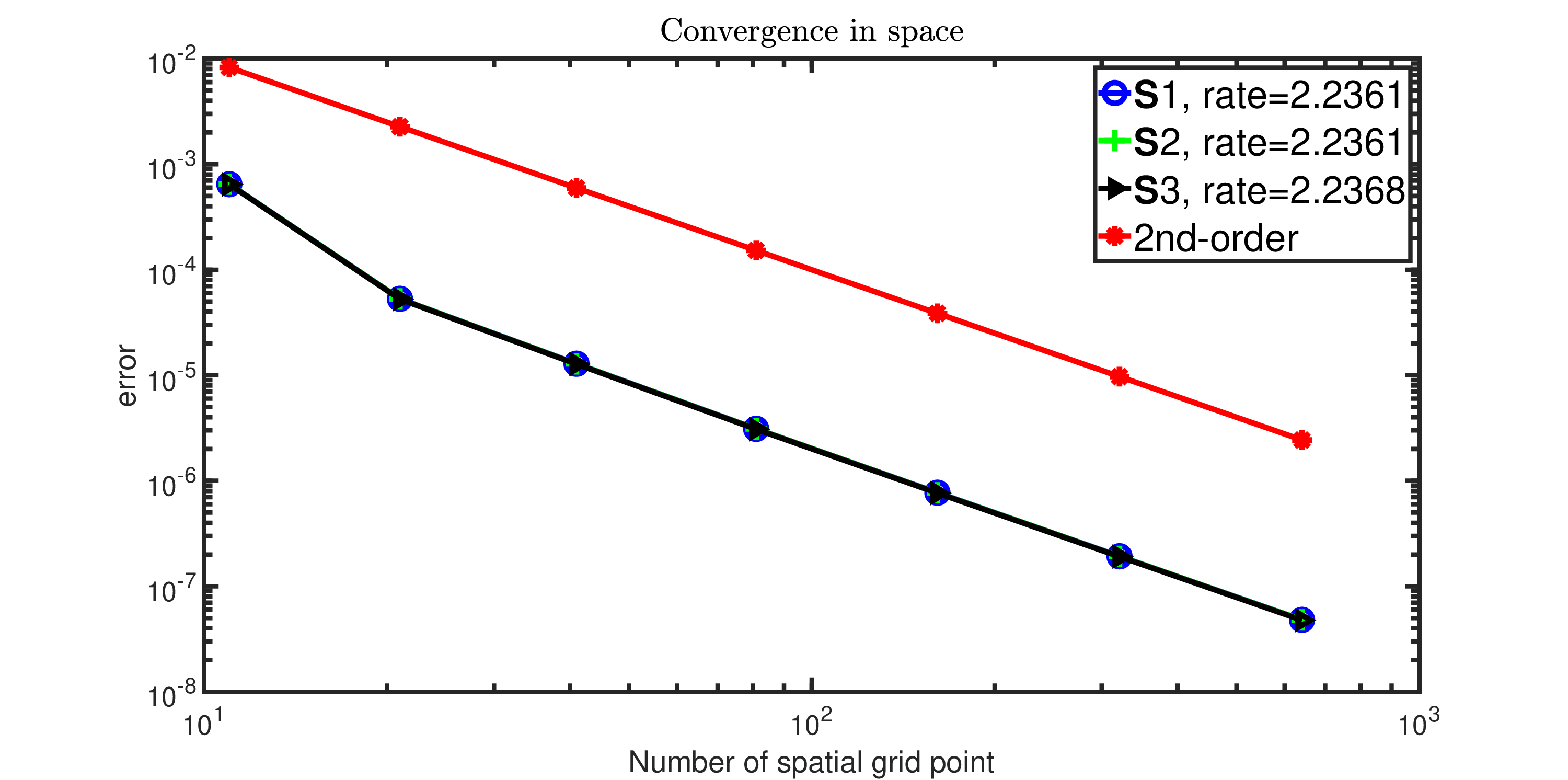} }}
    \caption{Convergence curves in mesh size in $\log-\log$ scale with $\epsilon=0.05,  \lambda=0.1$ : on the left for the state equation, on the right for the adjoint equation}
    \label{accuracy_space}
\end{figure}
Figure \ref{accuracy_space} shows how the suggested schemes accomplish the second order convergence in mesh size for both state and adjoint equations. 
\subsubsection{Solution for different target states} 
In this subsection, we plot solutions of the state equation and control for different desired state. For each experiment in this part, we use $\Omega=(0,1)$ with mesh size $h=1/256$, time step $\delta_t=1 e{-4}$ with $N_t=100$.
We run the scheme $\textbf{S}1$ to achieve the desired state $\widehat{y}(x, t)=\cos(2\pi x)$.
\begin{figure}[h!]
    \centering
    \subfloat{{\includegraphics[height=5.2cm,width=5cm]{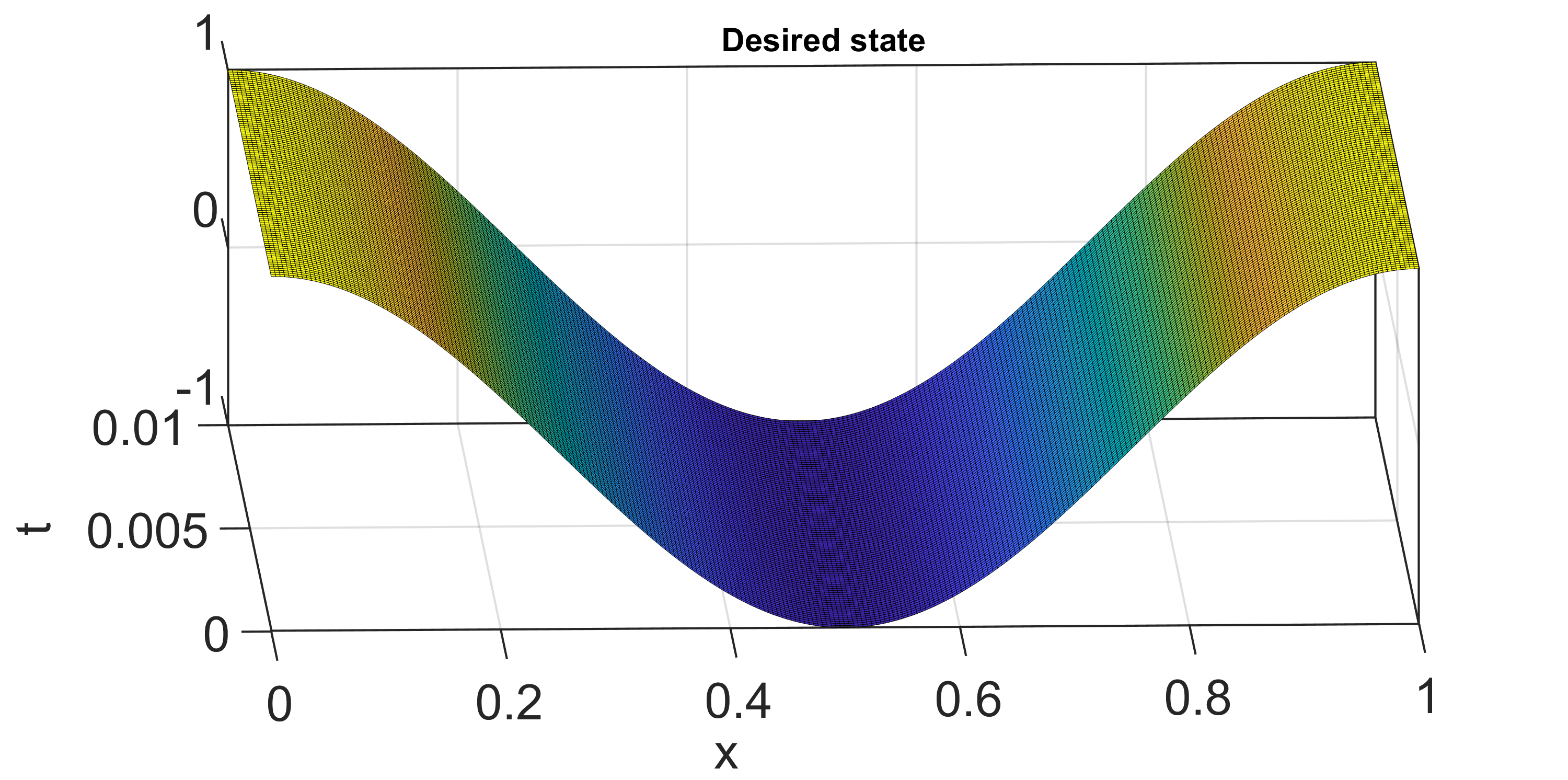} }}
     \subfloat{{\includegraphics[height=5.2cm,width=5cm]{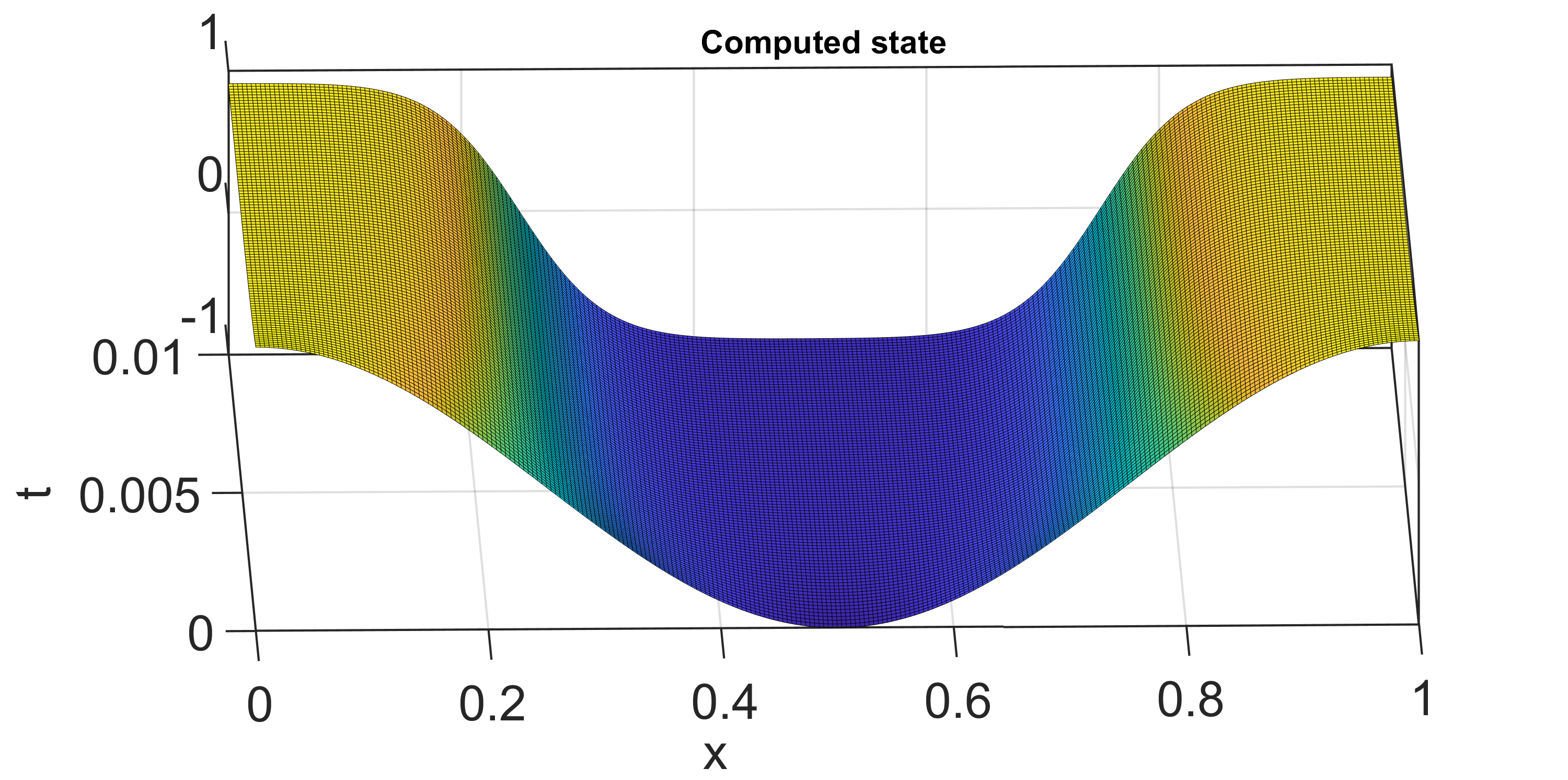} }}
     \subfloat{{\includegraphics[height=5.2cm,width=5cm]{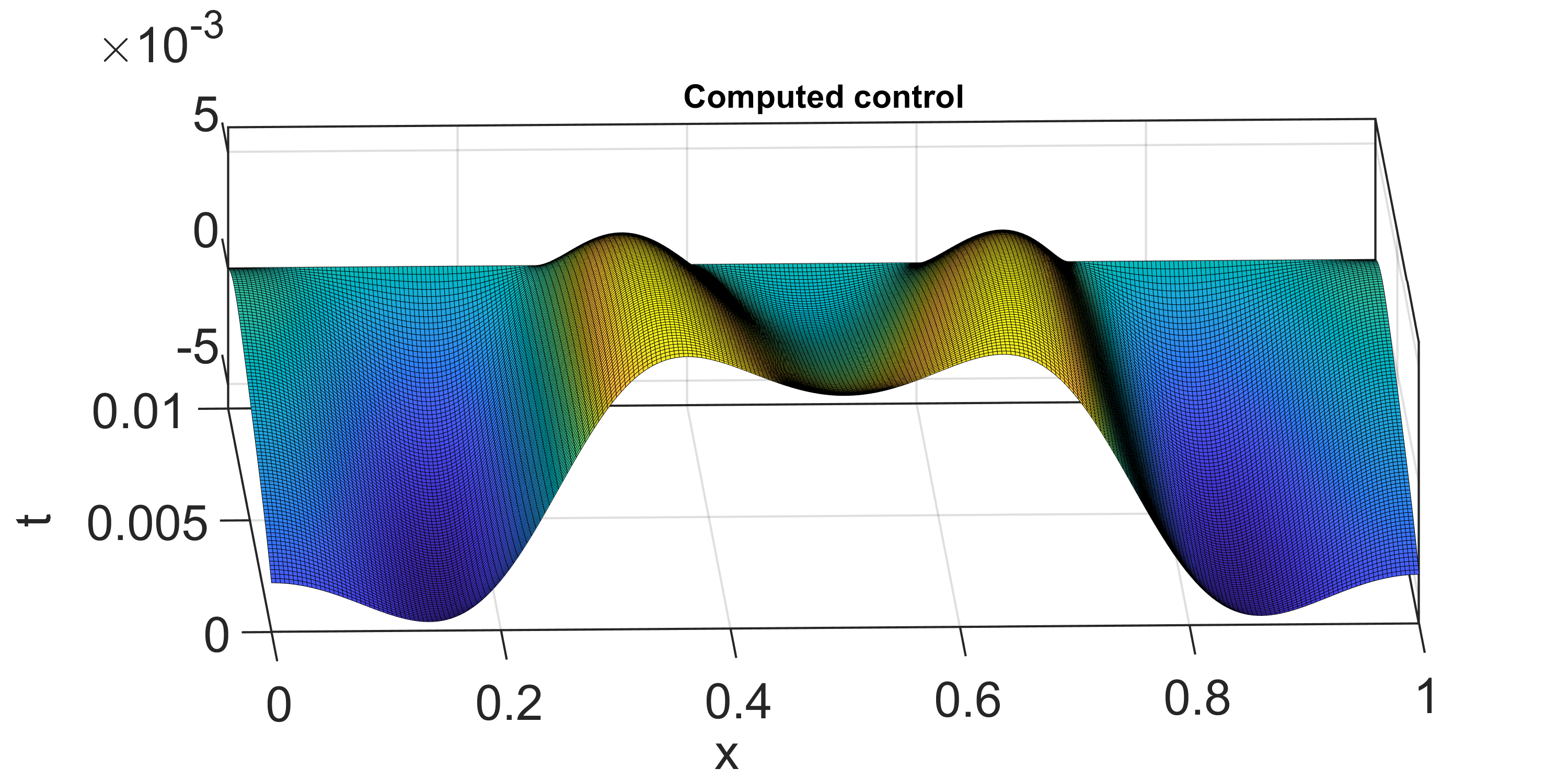} }}
    \caption{ On the left : the desired state; In the middle: the computed state; On the right : the computed control; For $\epsilon=0.05, \lambda=0.1$.}
    \label{s1}
\end{figure}
\begin{figure}[h!]
    \centering
     \subfloat{{\includegraphics[height=5.2cm,width=5.5cm]{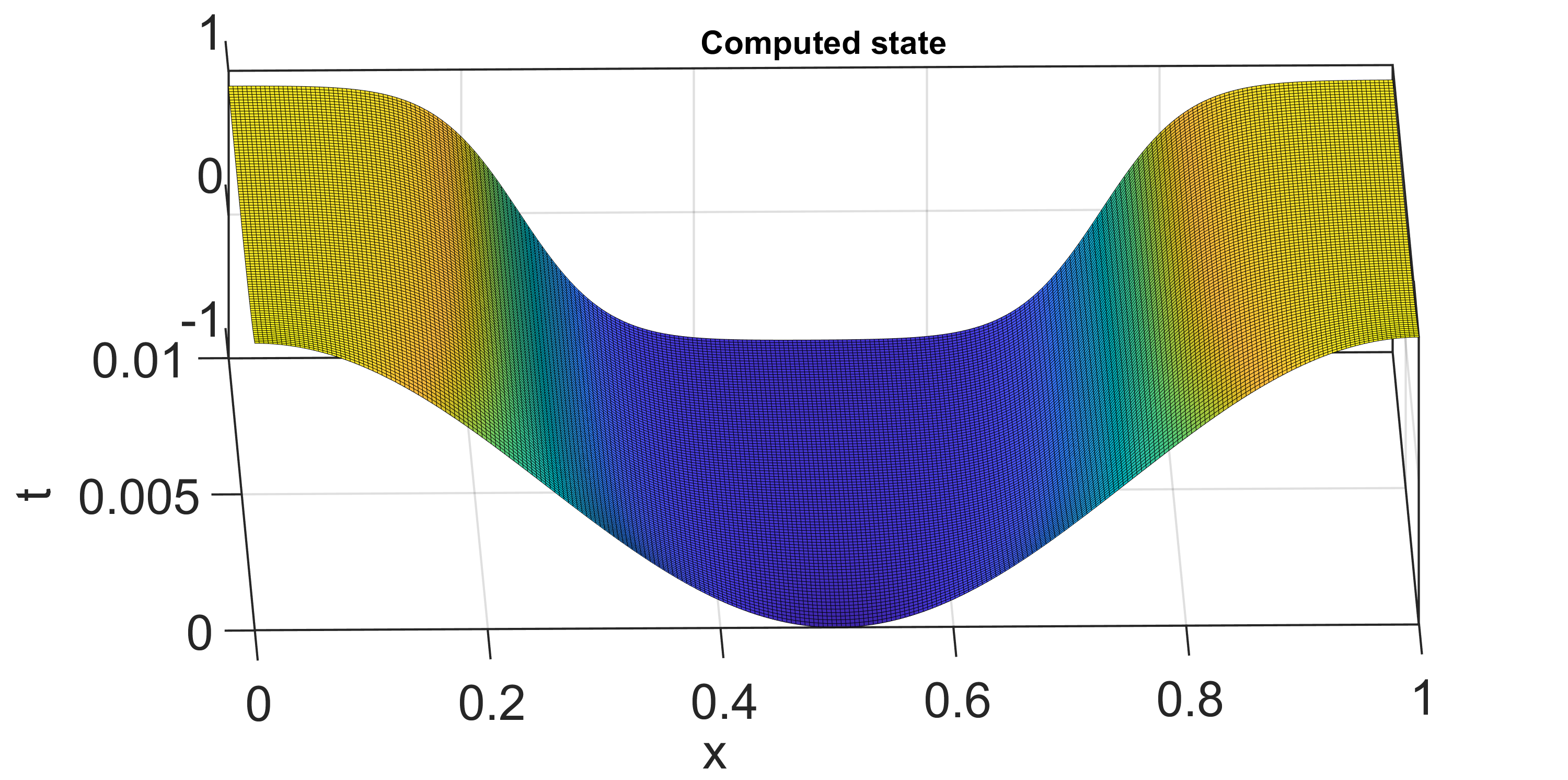} }}
     \subfloat{{\includegraphics[height=5.2cm,width=5.5cm]{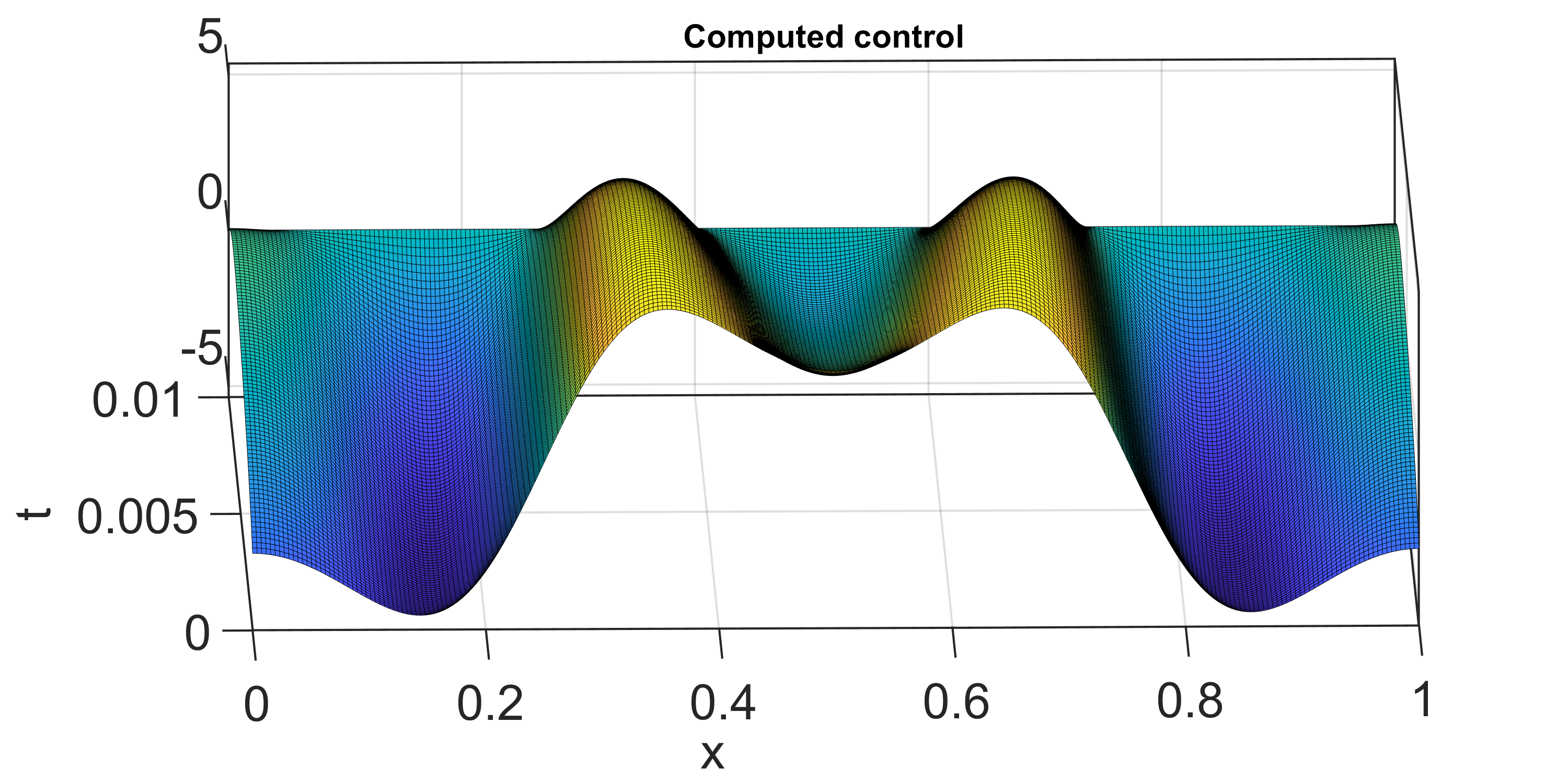} }}
    \caption{ On the left : the computed state; On the right : the computed control; For $\epsilon=0.05, \lambda=1 e{-4}$.}
    \label{s12}
\end{figure}
In Figure \ref{s1} we plot the desired state, the computed state and the optimal control for $\epsilon=0.05$ and $\lambda=0.1$. To get an improvement in accuracy on computed state we take $\lambda=1 e{-4}$ and plot the solution in Figure \ref{s12}. We note that just the magnitude of calculated control changes, with no improvement in computed state. Now, we test to check whether the calculated state improves by changing $\epsilon$ to $\epsilon=0.09$. The solutions for $\epsilon=0.09$ are shown in Figure 5, which demonstrates that the calculated state is extremely close to the target state.  So It is evident that in order to get a solution profile that is identical to the intended state, the calculated state relies on the choice of $\epsilon$.
\begin{figure}[h!]
    \centering
     \subfloat{{\includegraphics[height=5.2cm,width=5.5cm]{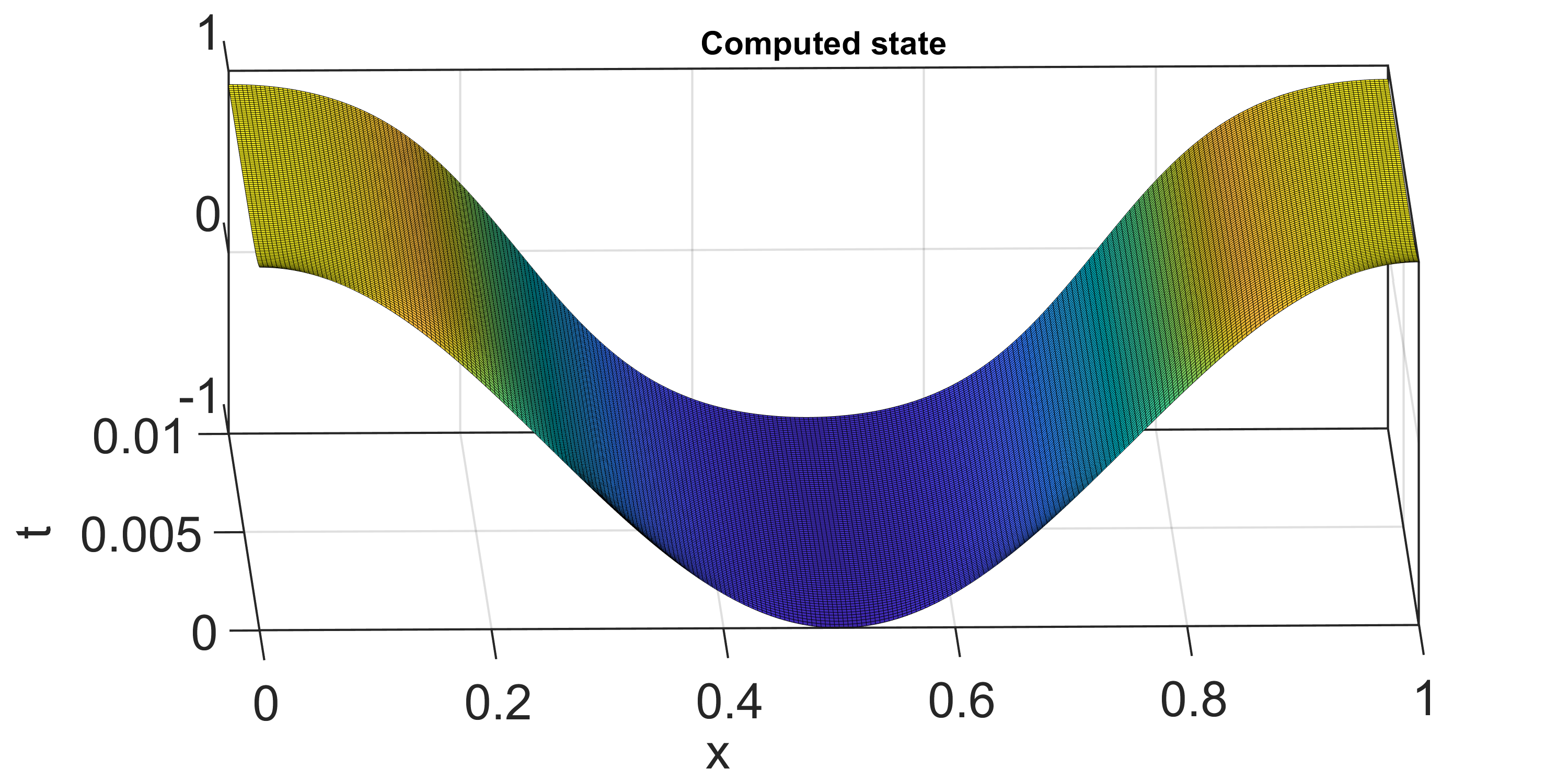} }}
    \subfloat{{\includegraphics[height=5.2cm,width=5.5cm]{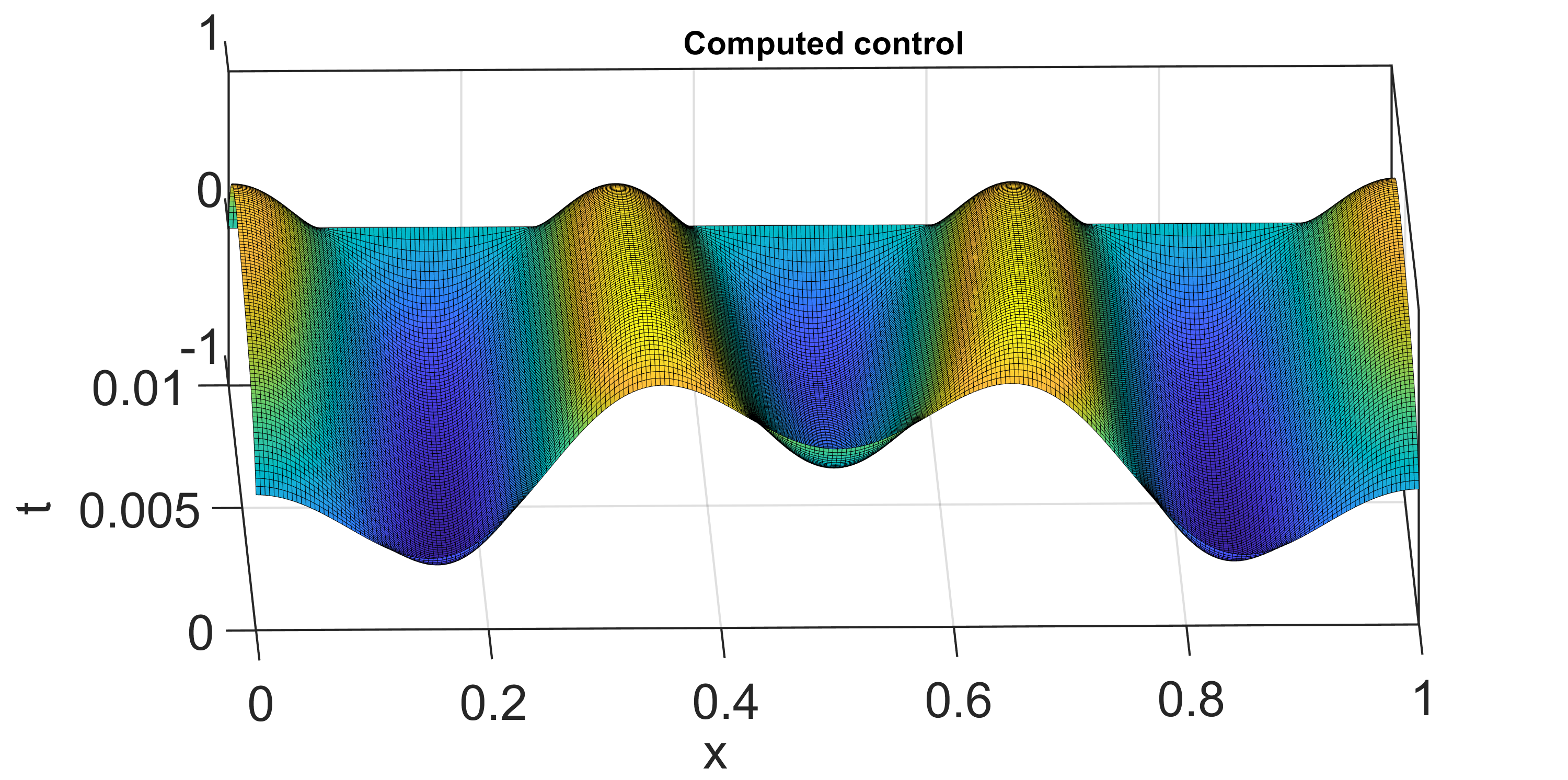} }}
    \caption{ On the left : the computed state; On the right : the computed control; For $\epsilon=0.09, \lambda=1 e{-4}$.}
    \label{s123}
\end{figure}
Next we run the scheme $\textbf{S}2$ with the target state $\widehat{y}(x,t)=\sin(\pi x)(t^2 +1)$.
\begin{figure}[h!]
    \centering
    \subfloat{{\includegraphics[height=5.2cm,width=5cm]{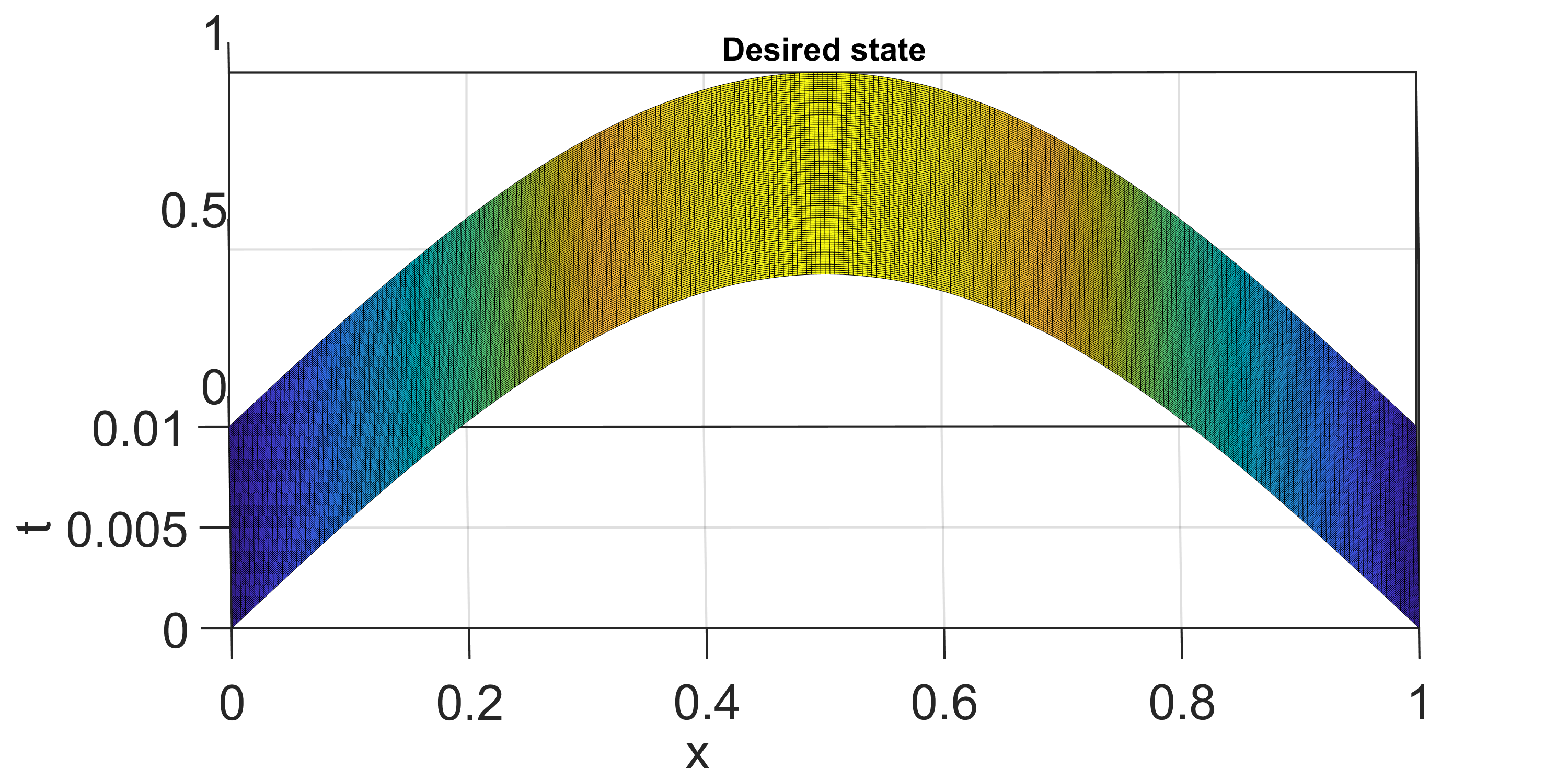} }}
     \subfloat{{\includegraphics[height=5.2cm,width=5cm]{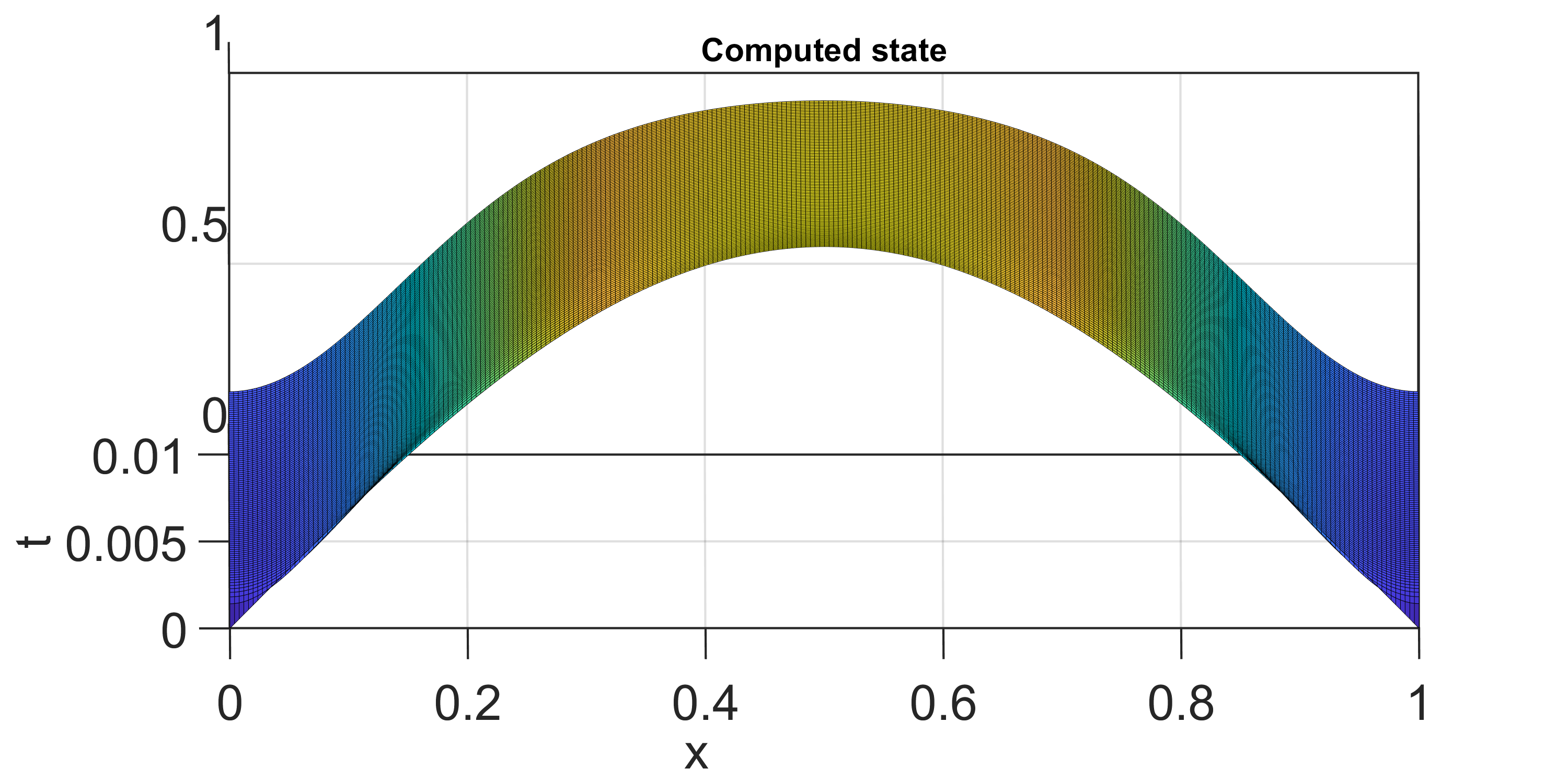} }}
     \subfloat{{\includegraphics[height=5.2cm,width=5cm]{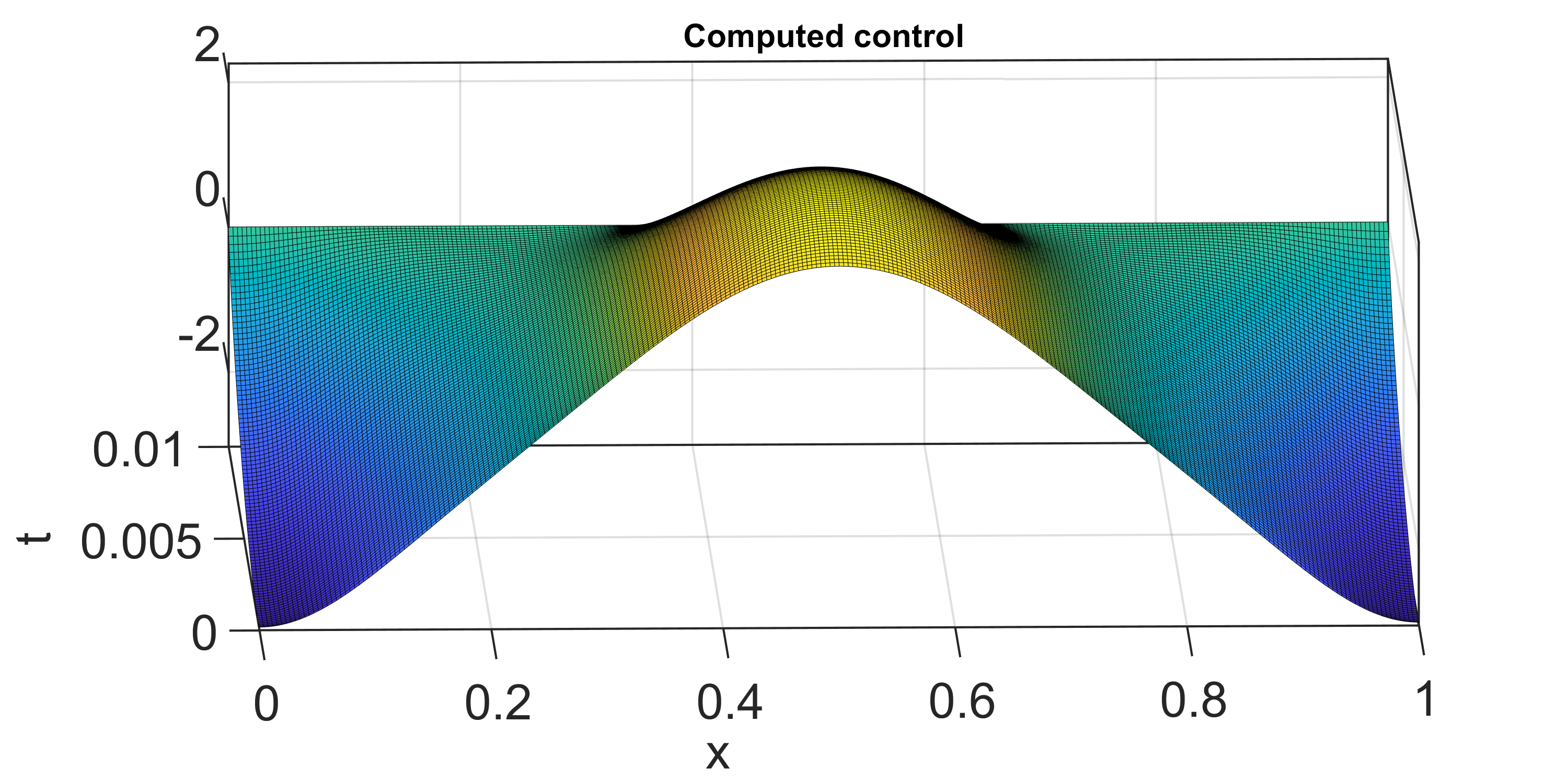} }}
    \caption{ On the left : the desired state; On the middle: the computed state; On the right : the computed control; for $\epsilon=0.09, \lambda=1 e{-4}$.}
    \label{s2}
\end{figure}
In Figure \ref{s2} we plot the desired state, the computed state and the optimal control for $\epsilon=0.09$ and $\lambda=0.0001$. It is evident that the calculated state is quite close to the target state. Lastly we run the final scheme $\textbf{S}3$ with the desired state $\widehat{y}(x,t)=0.1 e^{\cos(\pi x)} (3t^2 +1)$.  Figure \ref{s3} shows the target state, calculated state, and optimal control for $\epsilon=0.05$ and $\lambda=1 e{-4}$, with an acceptable degree of precision. 
\begin{figure}[h!]
    \centering
    \subfloat{{\includegraphics[height=5.2cm,width=5cm]{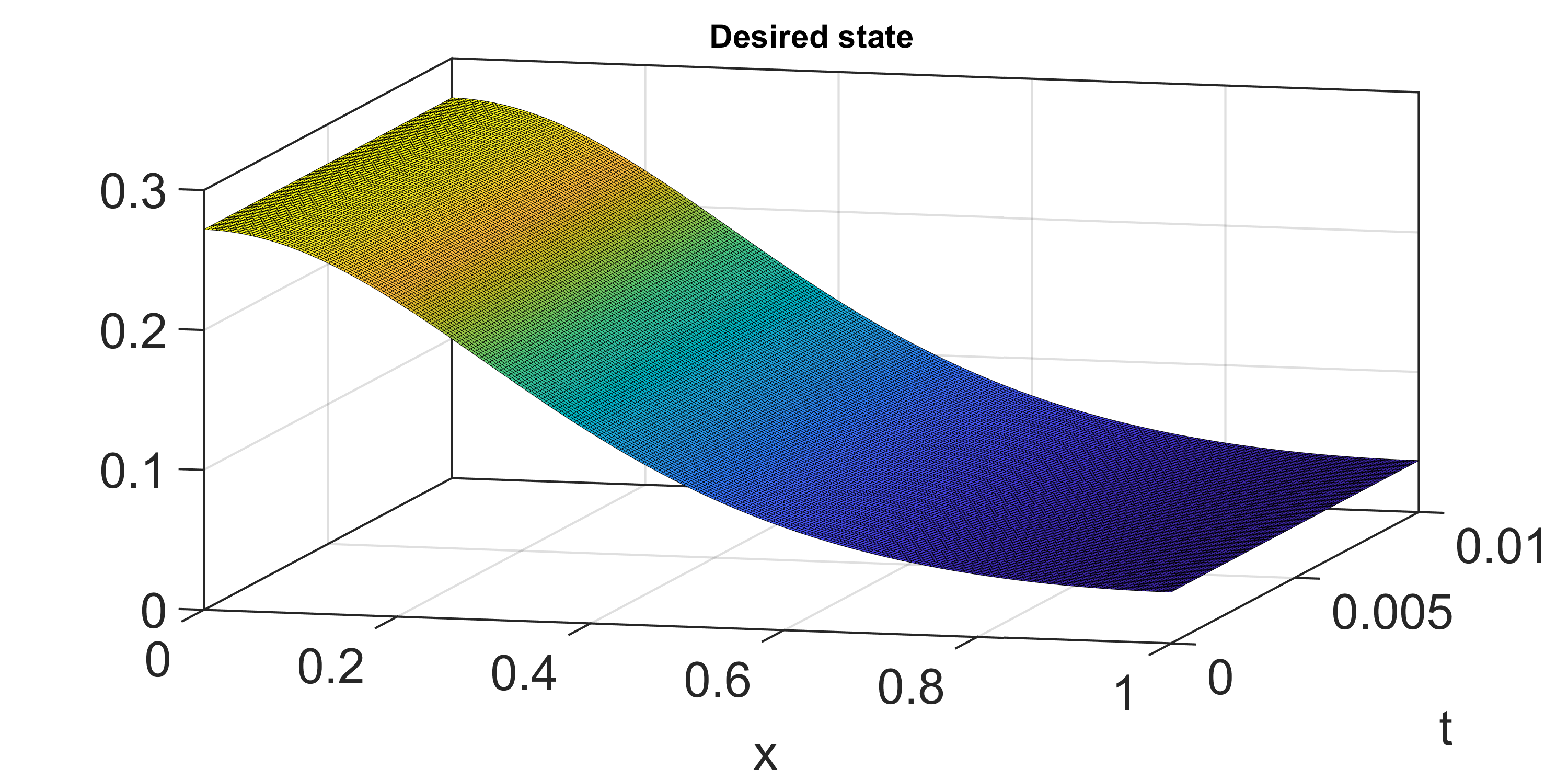} }}
     \subfloat{{\includegraphics[height=5.2cm,width=5cm]{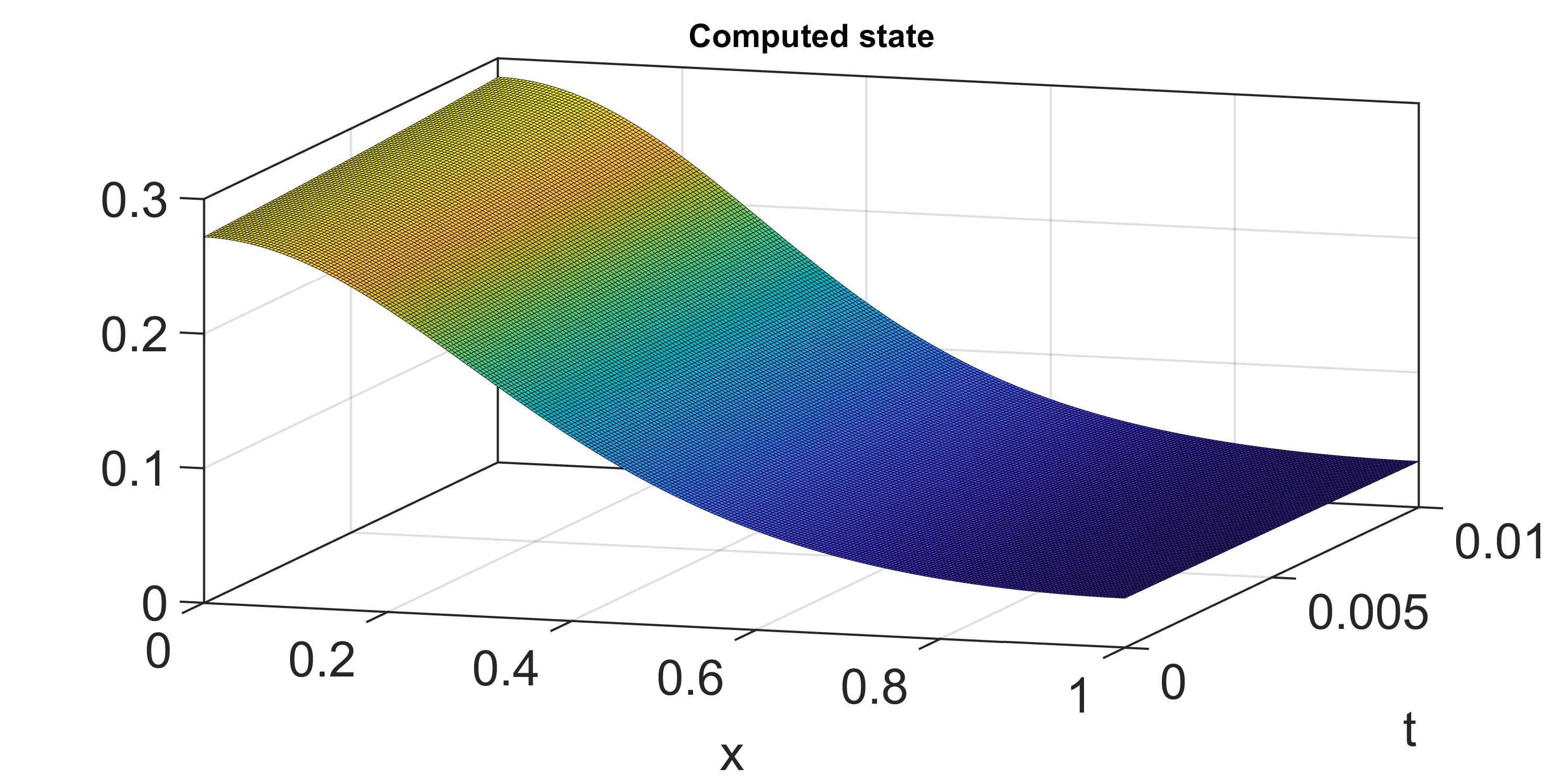} }}
     \subfloat{{\includegraphics[height=5.2cm,width=5cm]{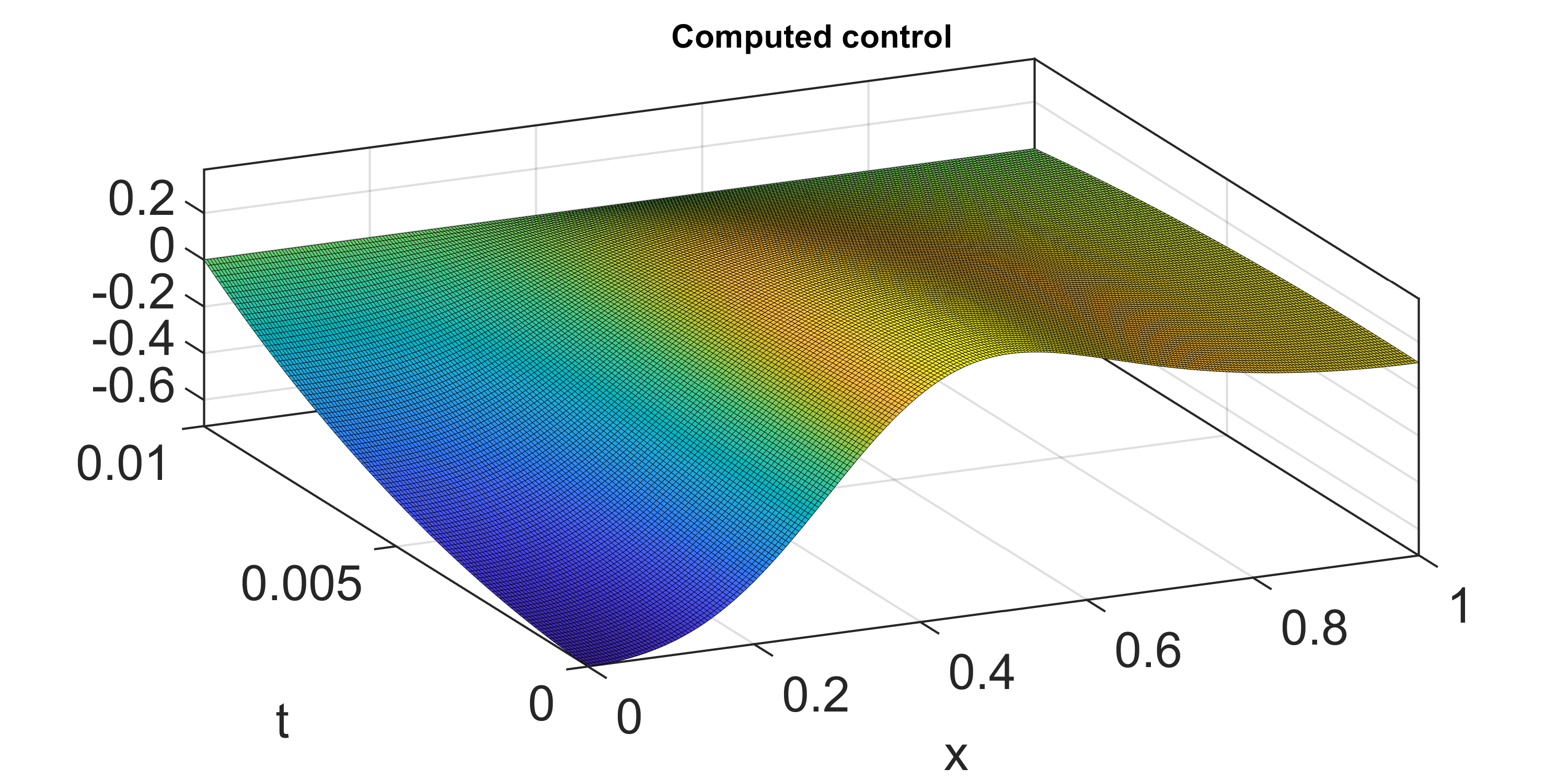} }}
    \caption{ On the left : the desired state; On the middle: the computed state; On the right : the computed control; For $\epsilon=0.05, \lambda=1 e{-4}$.}
    \label{s3}
\end{figure}
\subsection{Experiments in 2D}
For experiments in 2D we consider the spatial domain $\Omega=(0,1)^2$ with uniform mesh size $h=1/50$ on both direction for all the experiments in this subsection. First we run the scheme $\textbf{S}1$ with the desired state $\widehat{y}(x, t)=\cos(2\pi x y)$ over the time window $[0, T=0.01]$ with step size $\delta_t=5 e{-4}$.
\begin{figure}[h!]
    \centering
    \subfloat{{\includegraphics[height=5.5cm,width=7cm]{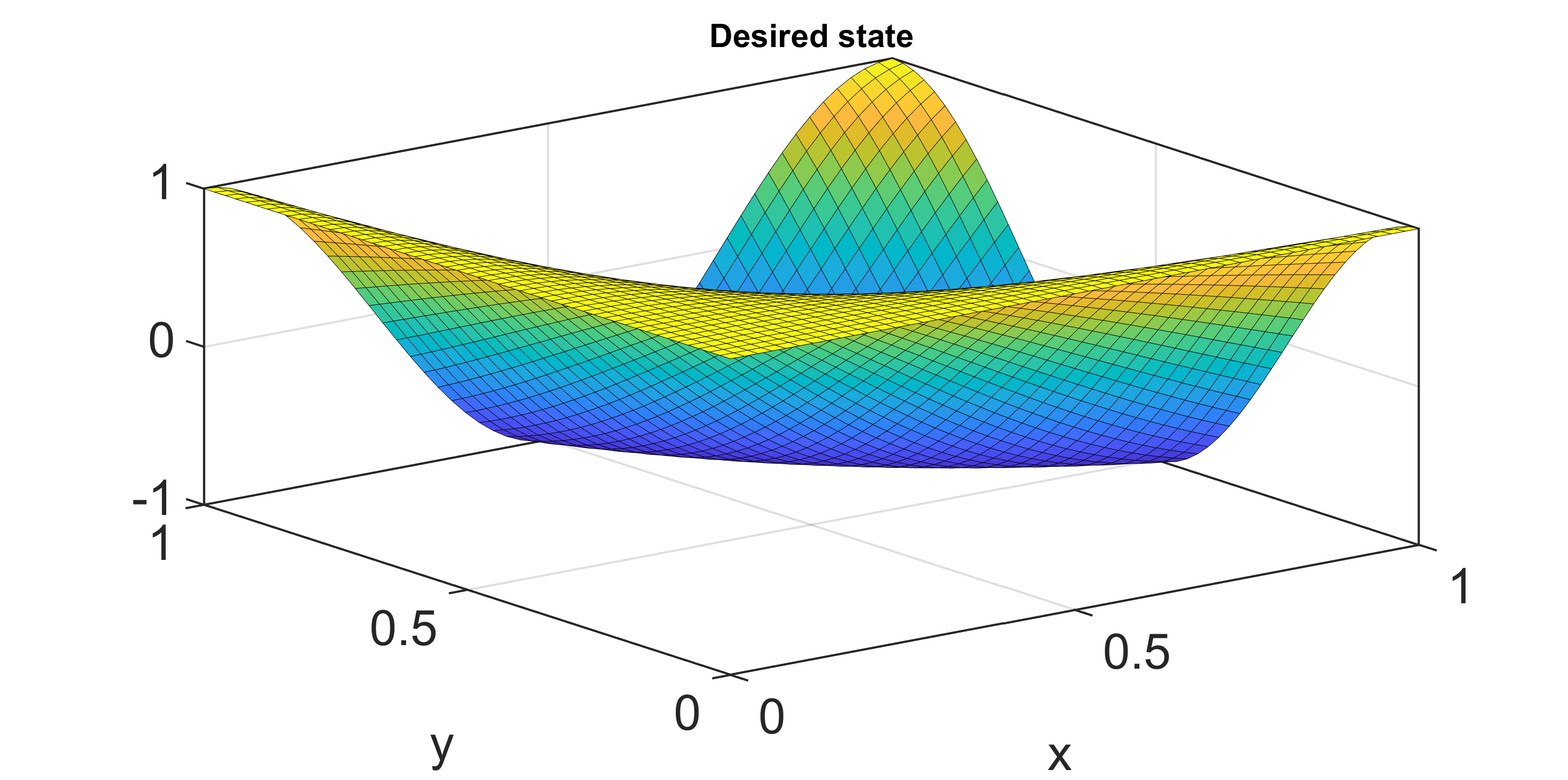} }}
     \subfloat{{\includegraphics[height=5.5cm,width=7cm]{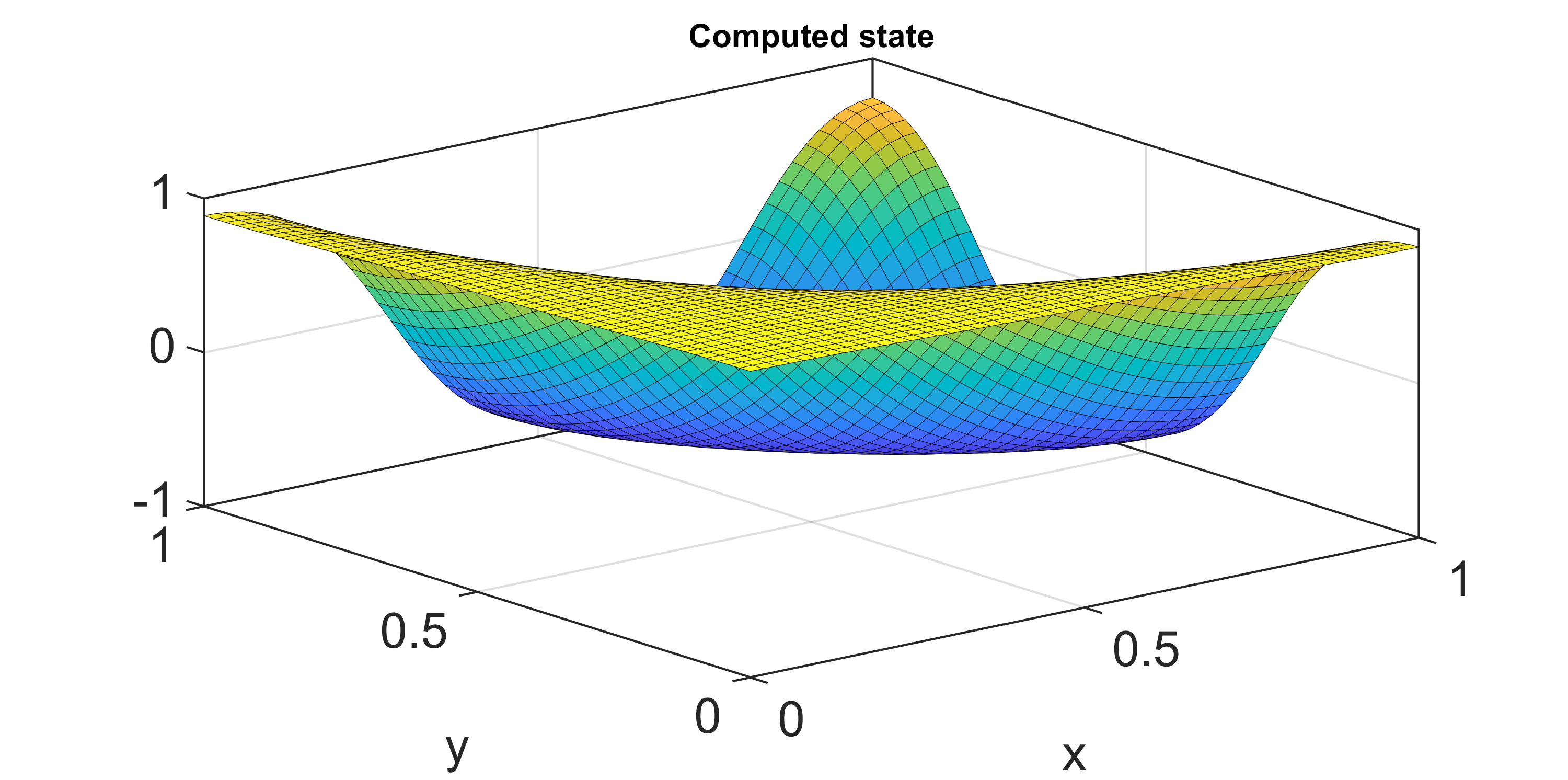} }}
    \caption{ On the left : the desired state; On the right : the computed state; for $\epsilon=0.1, \lambda=1 e{-4}$.}
    \label{s1_2d_state}
\end{figure}
\begin{figure}[h!]
    \centering
    \subfloat{{\includegraphics[height=5.5cm,width=7cm]{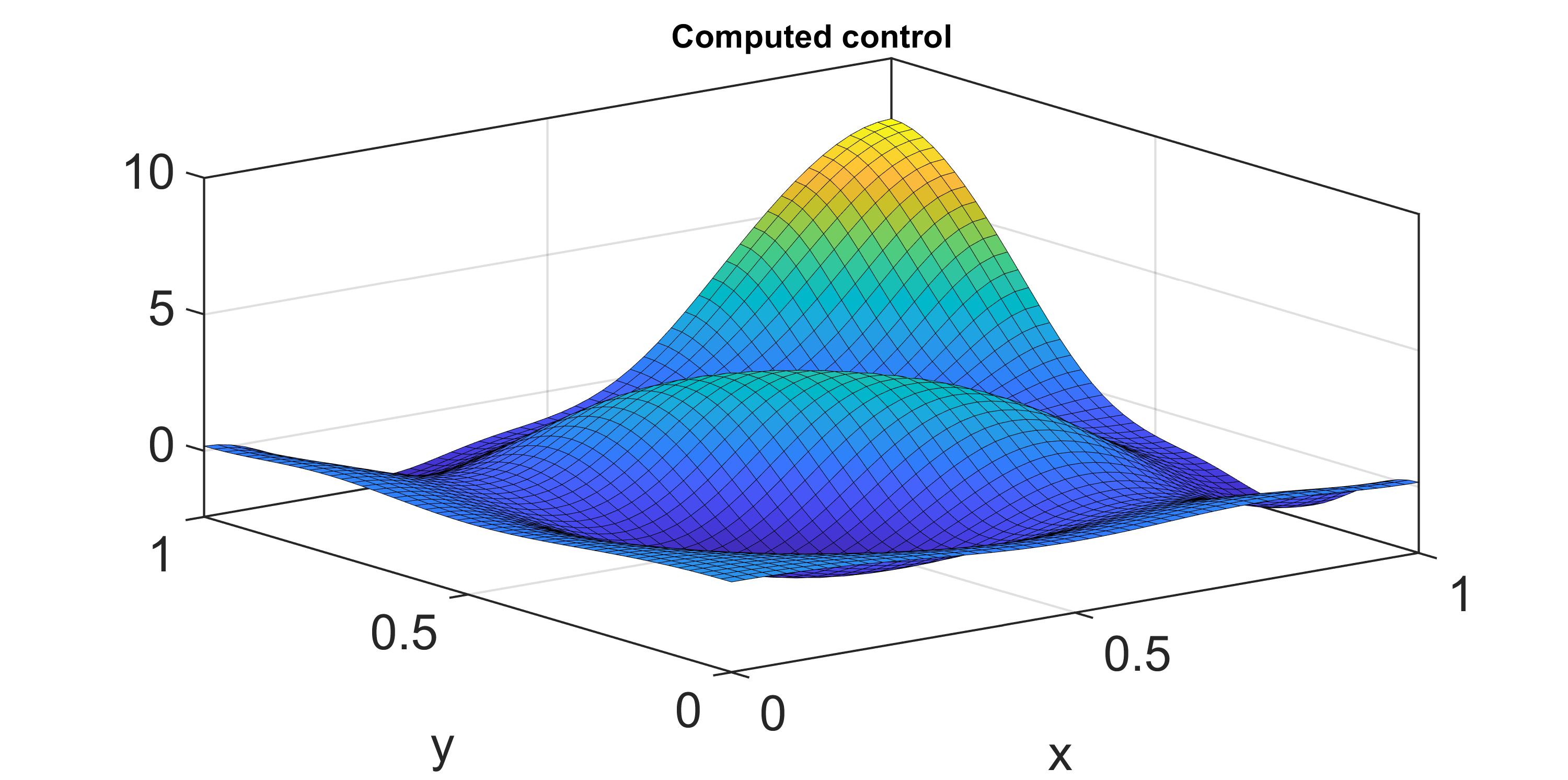} }}
     \subfloat{{\includegraphics[height=5.5cm,width=7cm]{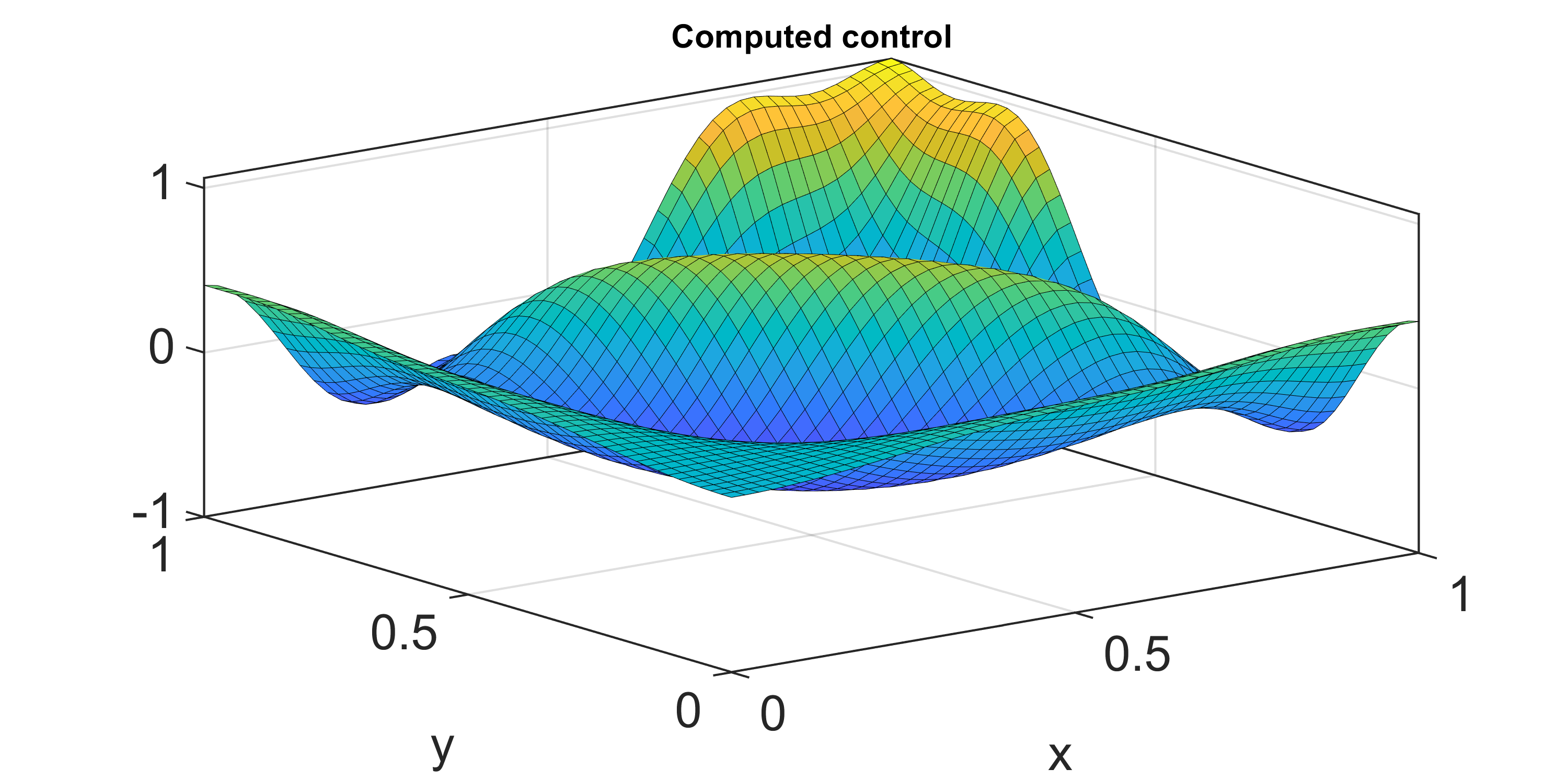} }}
    \caption{ On the left : the computed control at $t=0$; On the right : the computed control at $t=T-\delta_t$; for $\epsilon=0.1, \lambda=1 e{-4}$.}
    \label{s1_2d_control}
\end{figure}
\begin{figure}[h!]
    \centering
    \subfloat{{\includegraphics[height=5.5cm,width=7cm]{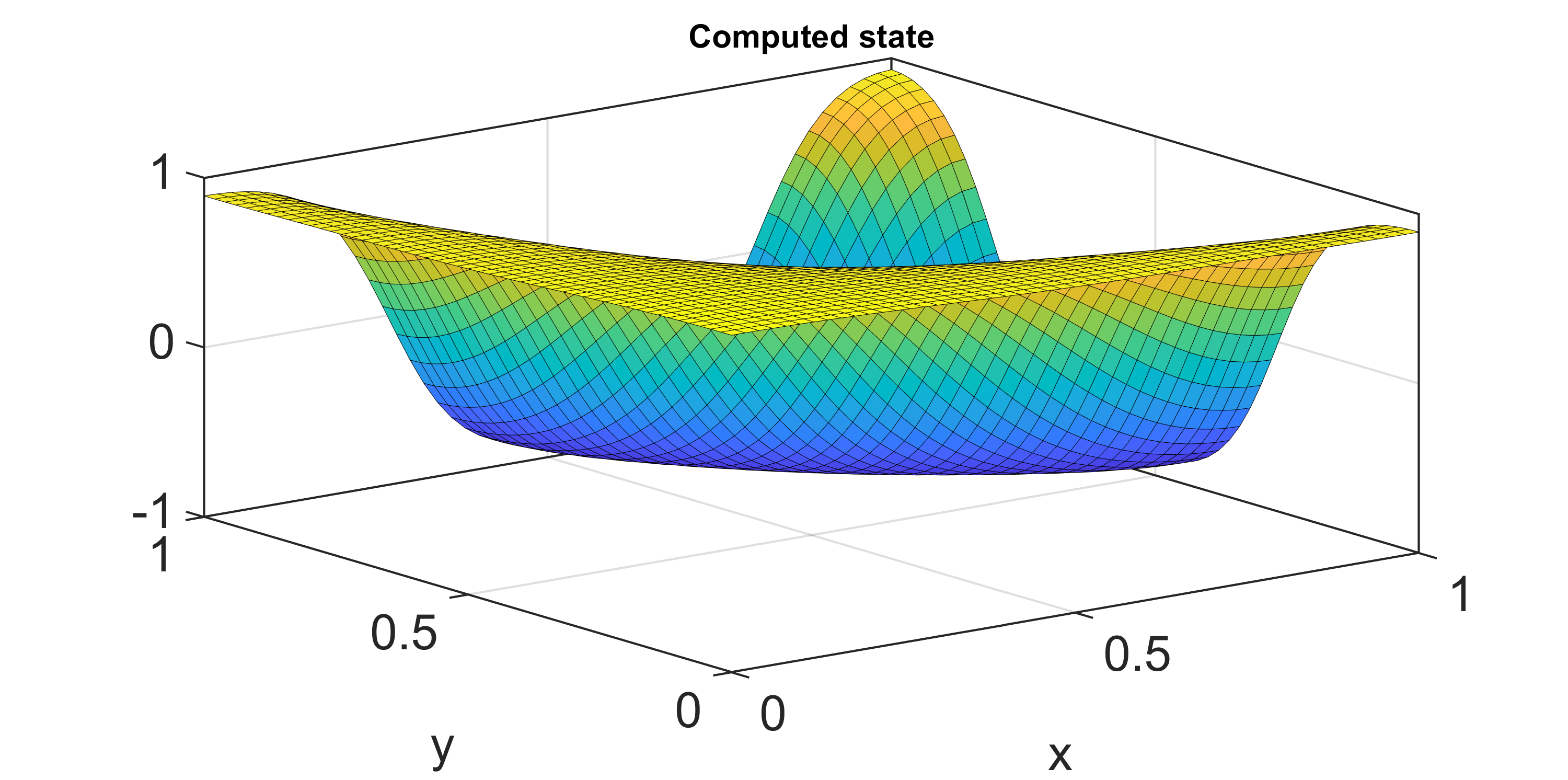} }}
     \subfloat{{\includegraphics[height=5.5cm,width=7cm]{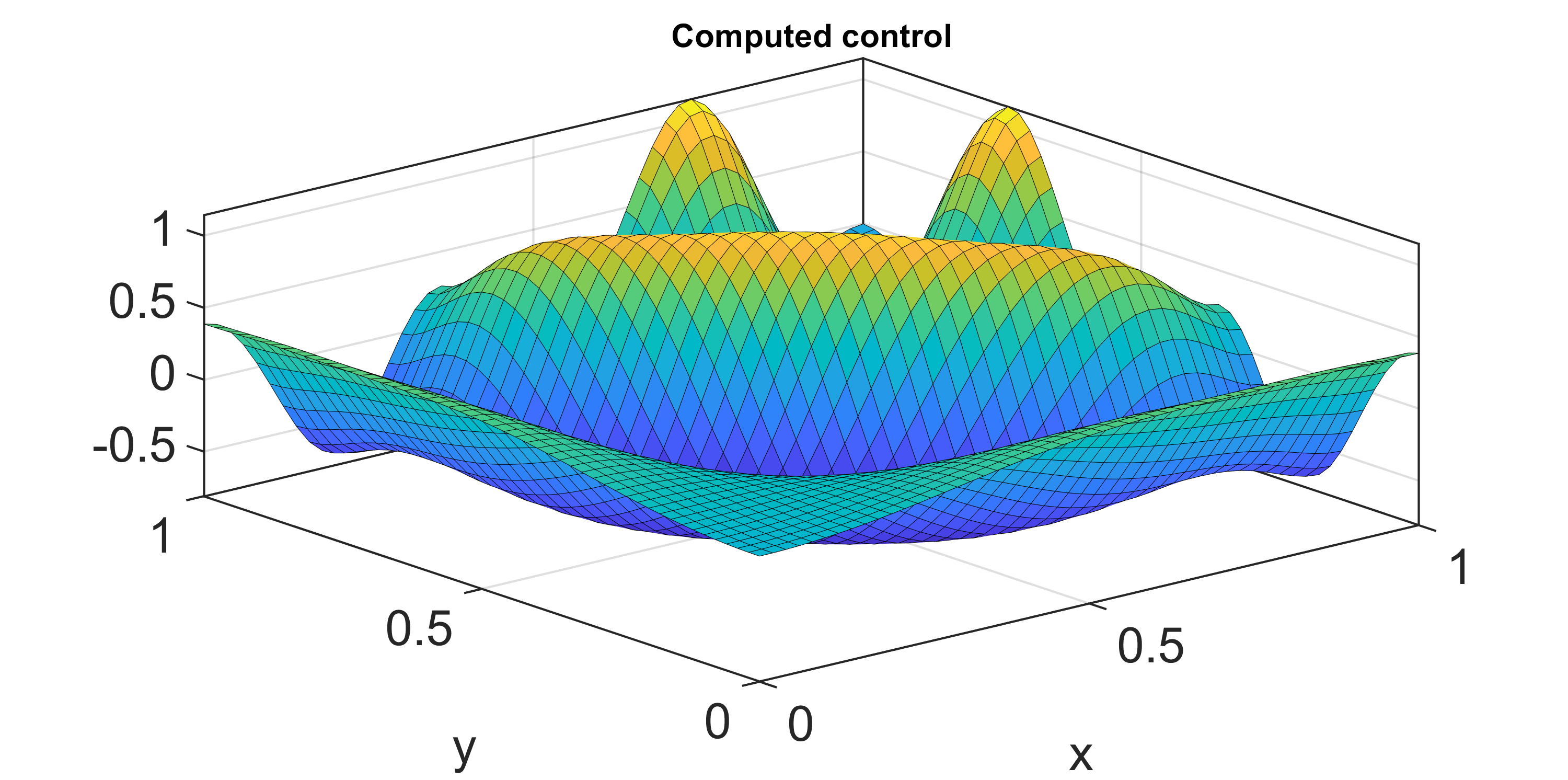} }}
    \caption{ On the left : the computed state at $t=T$; On the right : the computed control at $t=T-\delta_t$; for $\epsilon=0.07, \lambda=1 e{-4}$.}
    \label{s1_2d_control_state}
\end{figure}
We illustrate the intended state and calculated state in Figure \ref{s1_2d_state}.  We depict the calculated control at $t=0$ and $t=T-\delta_t$ in Figure  \ref{s1_2d_control}. We use $\epsilon=0.07$ to enhance the calculated state, as seen in the left panel of Figure \ref{s1_2d_control_state}; we additionally draw the corresponding control at $t=T-\delta_t$.

The scheme $\textbf{S}2$ is then executed to obtain the desired state $\widehat{y}(x, t)=0.5\cos(2\pi x) \cos(2\pi y)(e^{-0.1 t}+2t)$ over the time window $[0, T=0.03]$ with time step $\delta_t=1 e{-3} $
\begin{figure}[h!]
    \centering
    \subfloat{{\includegraphics[height=5.5cm,width=7cm]{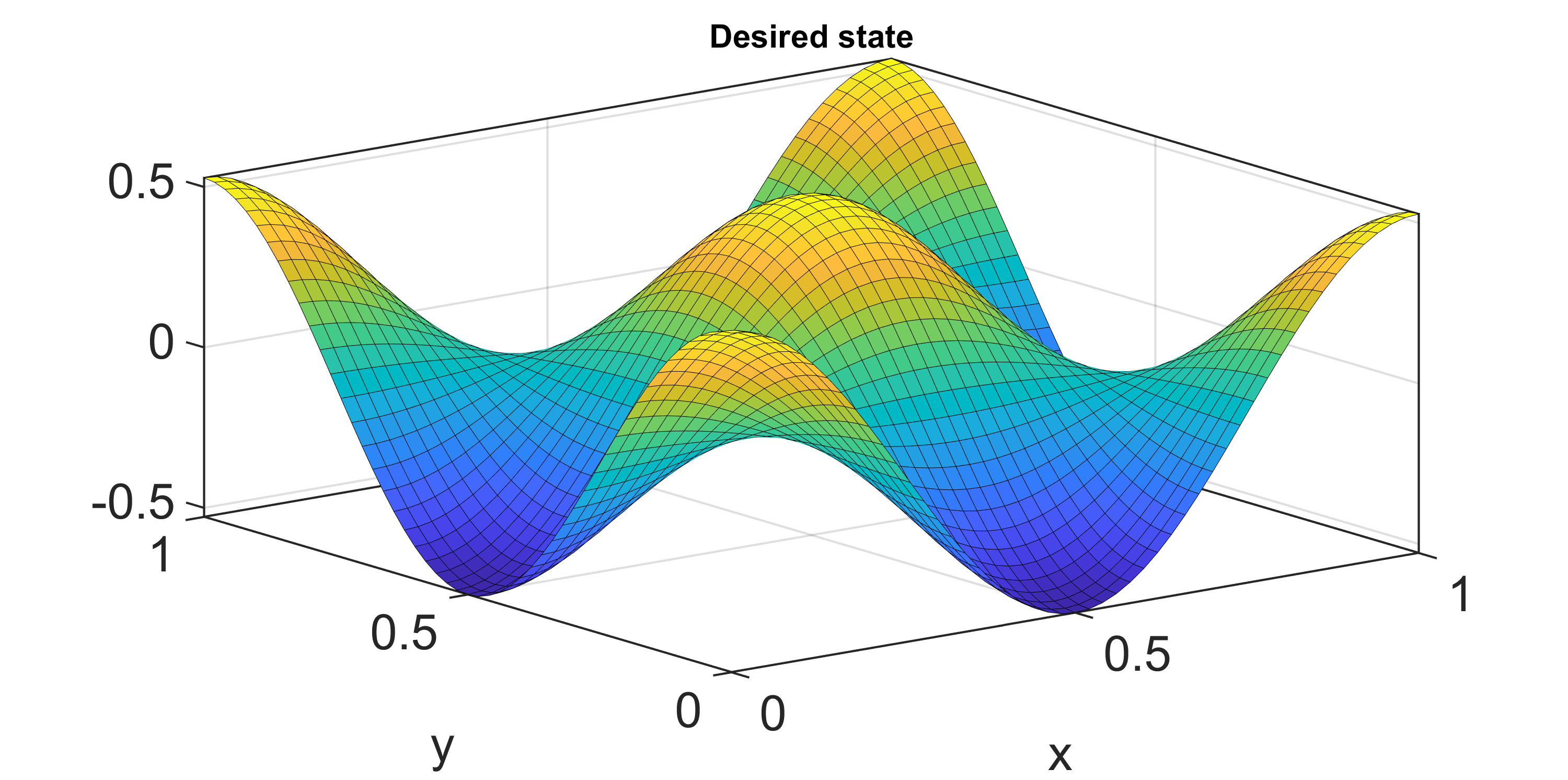} }}
     \subfloat{{\includegraphics[height=5.5cm,width=7cm]{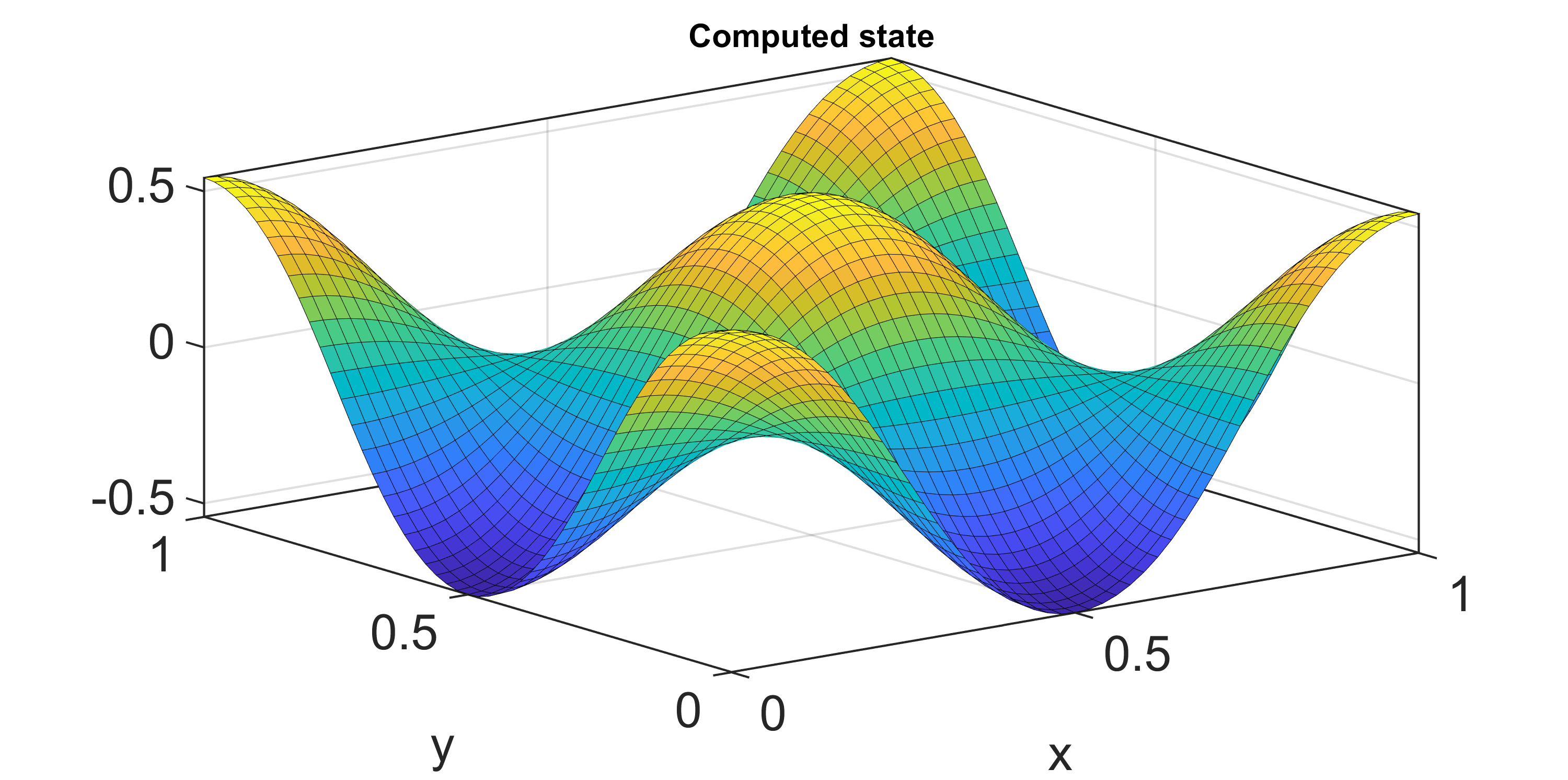} }}
    \caption{ On the left : the desired state; On the right : the computed state; for $\epsilon=0.1, \lambda=0.01$.}
    \label{s2_2d_state}
\end{figure} 
\begin{figure}[h!]
    \centering
    \subfloat{{\includegraphics[height=5.5cm,width=7cm]{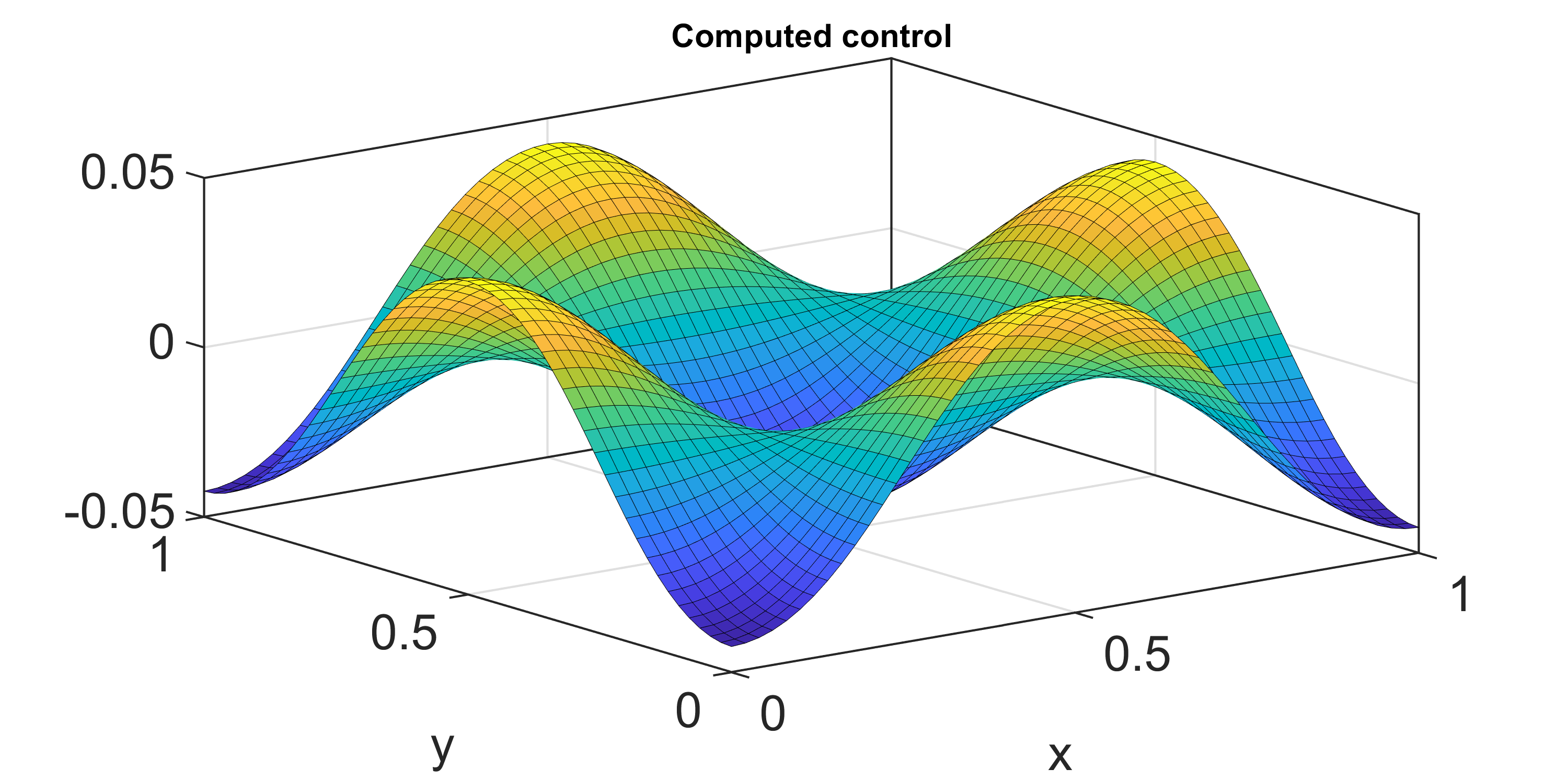} }}
     \subfloat{{\includegraphics[height=5.5cm,width=7cm]{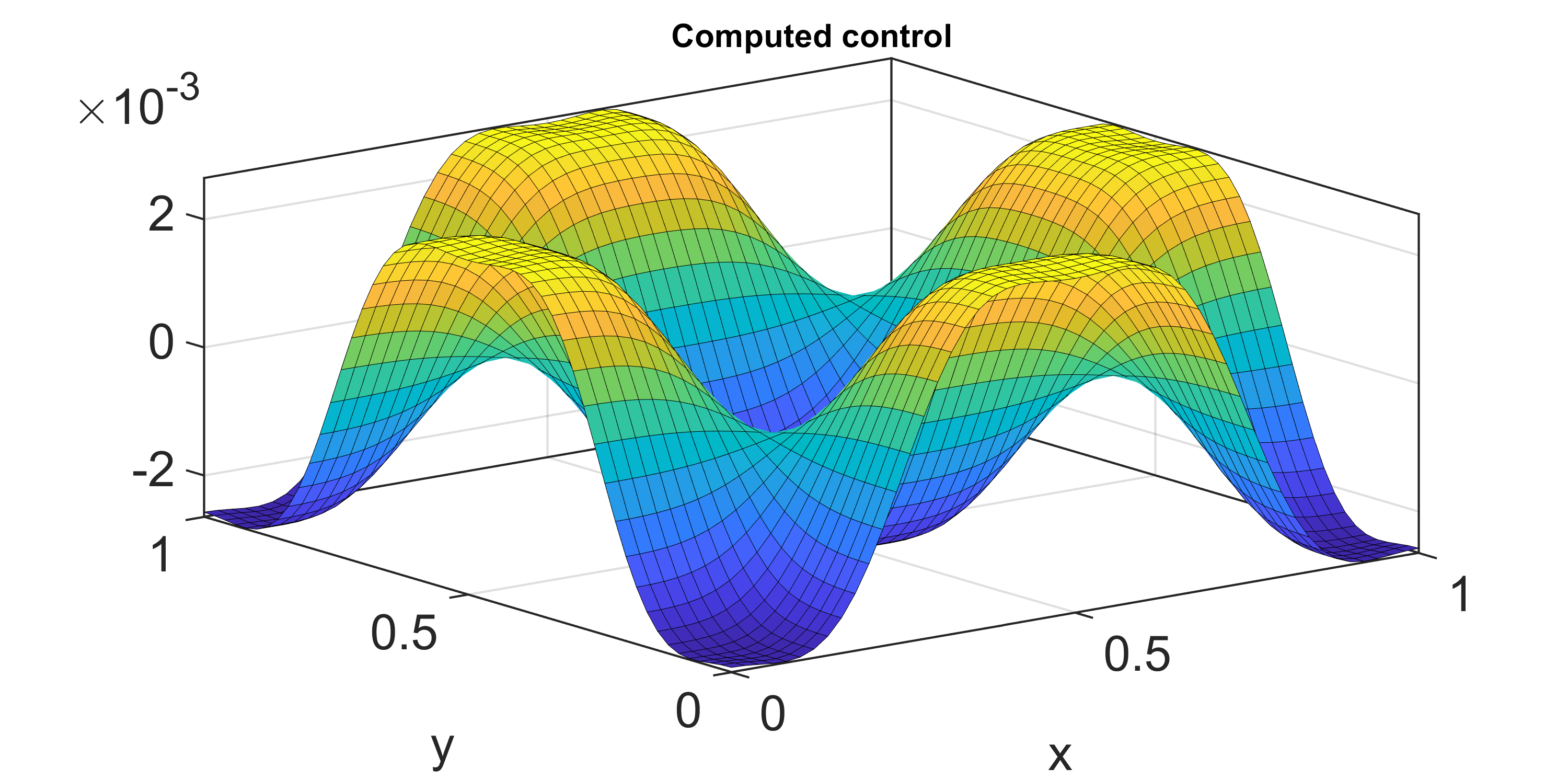} }}
    \caption{ On the left : computed control at $t=0$; On the right :  computed control at $t=T-\delta_t$; for $\epsilon=0.1, \lambda=0.01$.}
    \label{s2_2d_control}
\end{figure}
\begin{figure}[h!]
    \centering
    \subfloat{{\includegraphics[height=5.5cm,width=7cm]{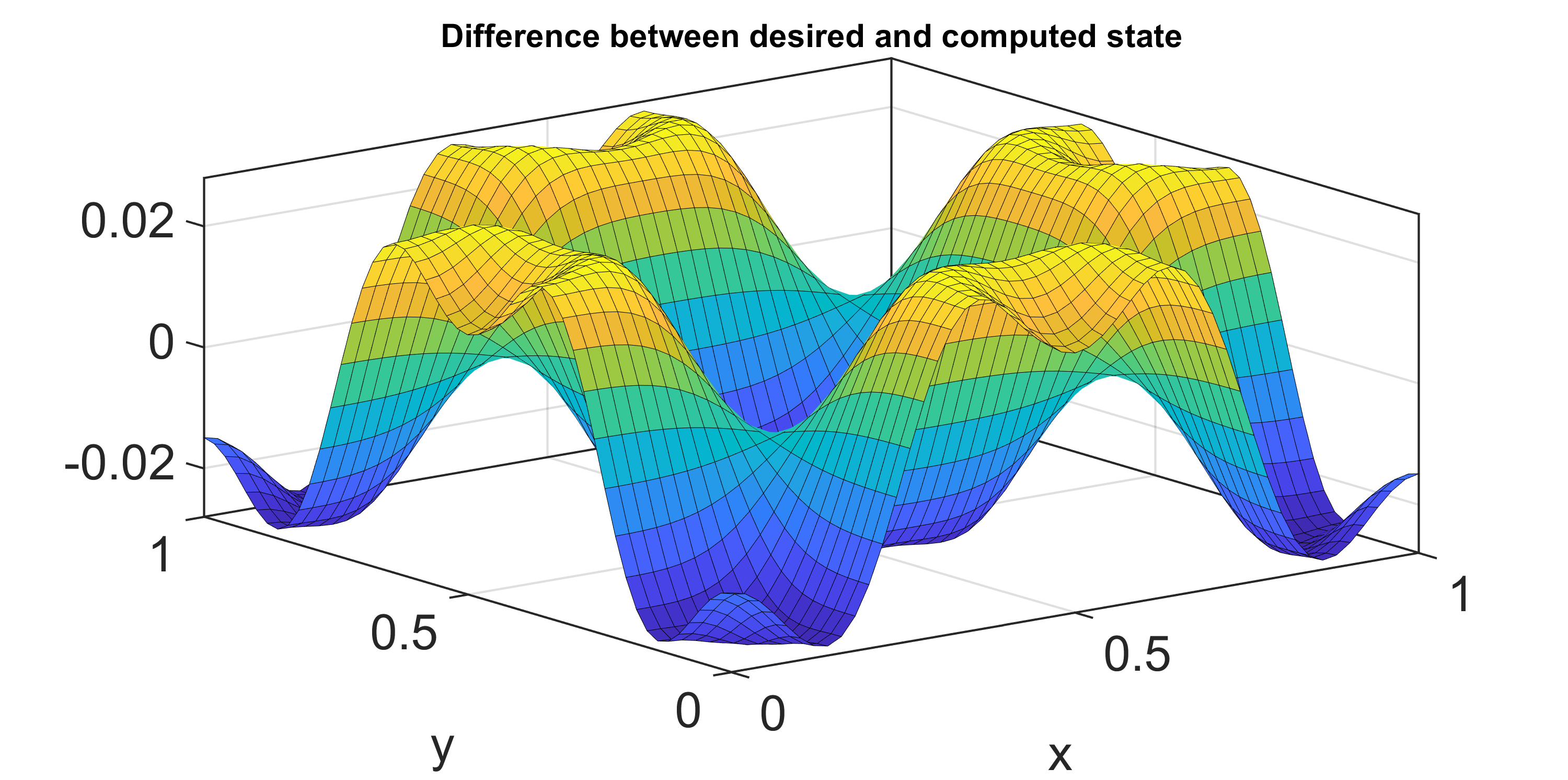} }}
    
    \caption{ Difference between state and control at $t=T$ for $\epsilon=0.1, \lambda=0.01$.}
    \label{s2_2d_diff}
\end{figure}
The final target state and calculated state are shown in Figure \ref{s2_2d_state}, and it is evident from the plot that the two plots can not be distinguished. At $t=0$ and $t=T-\delta_t$ in Figure 2\ref{s2_2d_control} we display the corresponding control. In Figure \ref{s2_2d_diff}, we exhibit the difference between the two states at the final time to show how near the calculated state is to the desired state. The difference, as we can see, is in the order of $10^{-2}$. Keep in mind that by selecting the relatively larger values of $\epsilon=0.1$ and $\lambda=0.01$, we are able to achieve the target state with an absolute difference of order $10^{-2}$.  At this point, we may thus conclude that, in addition to the choices of $\epsilon$ and $ \lambda$,  the target state choice also has an impact.

The scheme $\textbf{S}3$ is now executed to attain the desired state $\widehat{y}(x, t)=0.5\sin(x y)(1-2t)$ across the time frame $[0, T=0.01]$ with a step sige of $\delta_t=3.33 e{-4}$.

\begin{figure}[h!]
    \centering
    \subfloat{{\includegraphics[height=5.5cm,width=7cm]{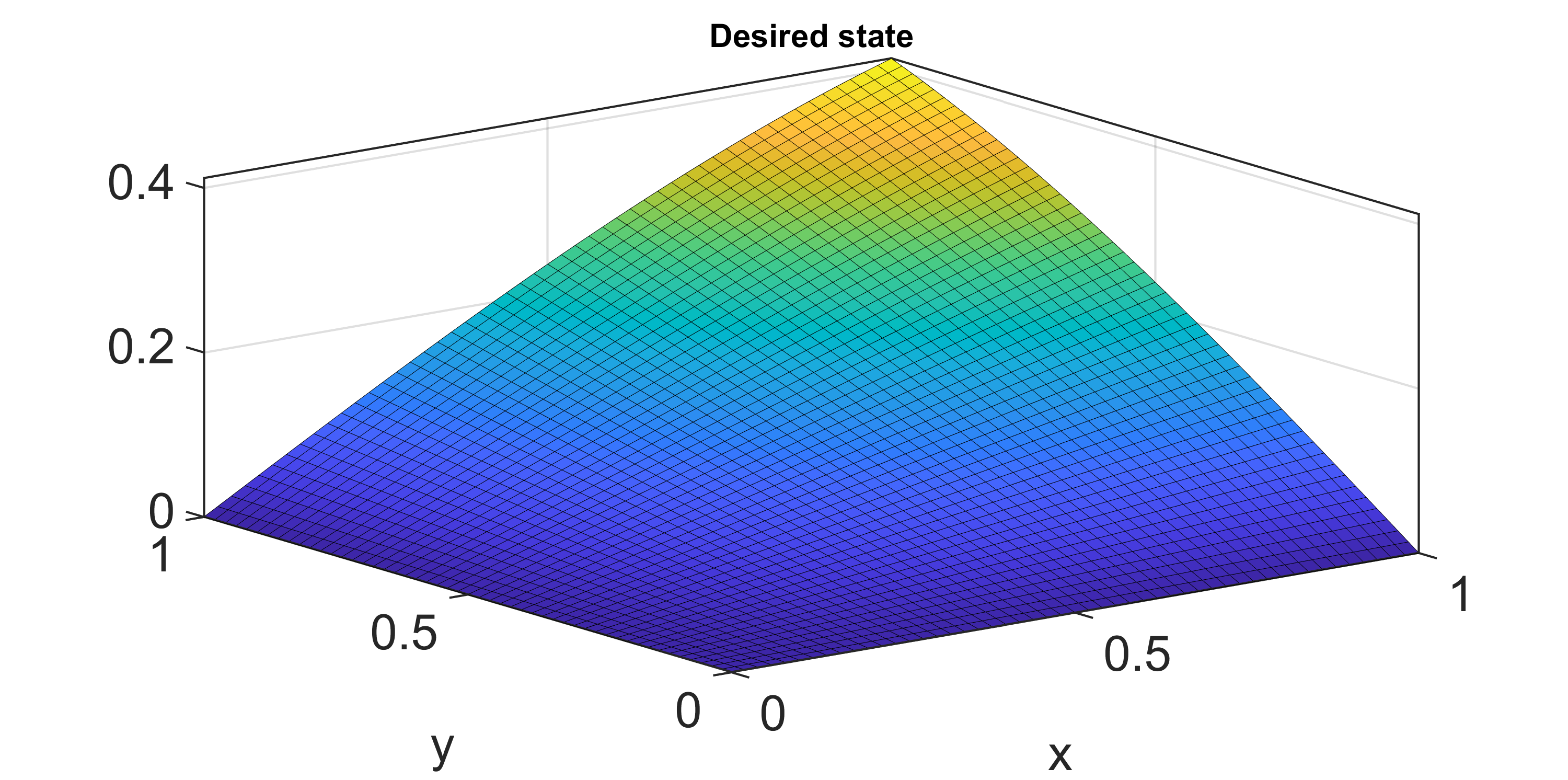} }}
     \subfloat{{\includegraphics[height=5.5cm,width=7cm]{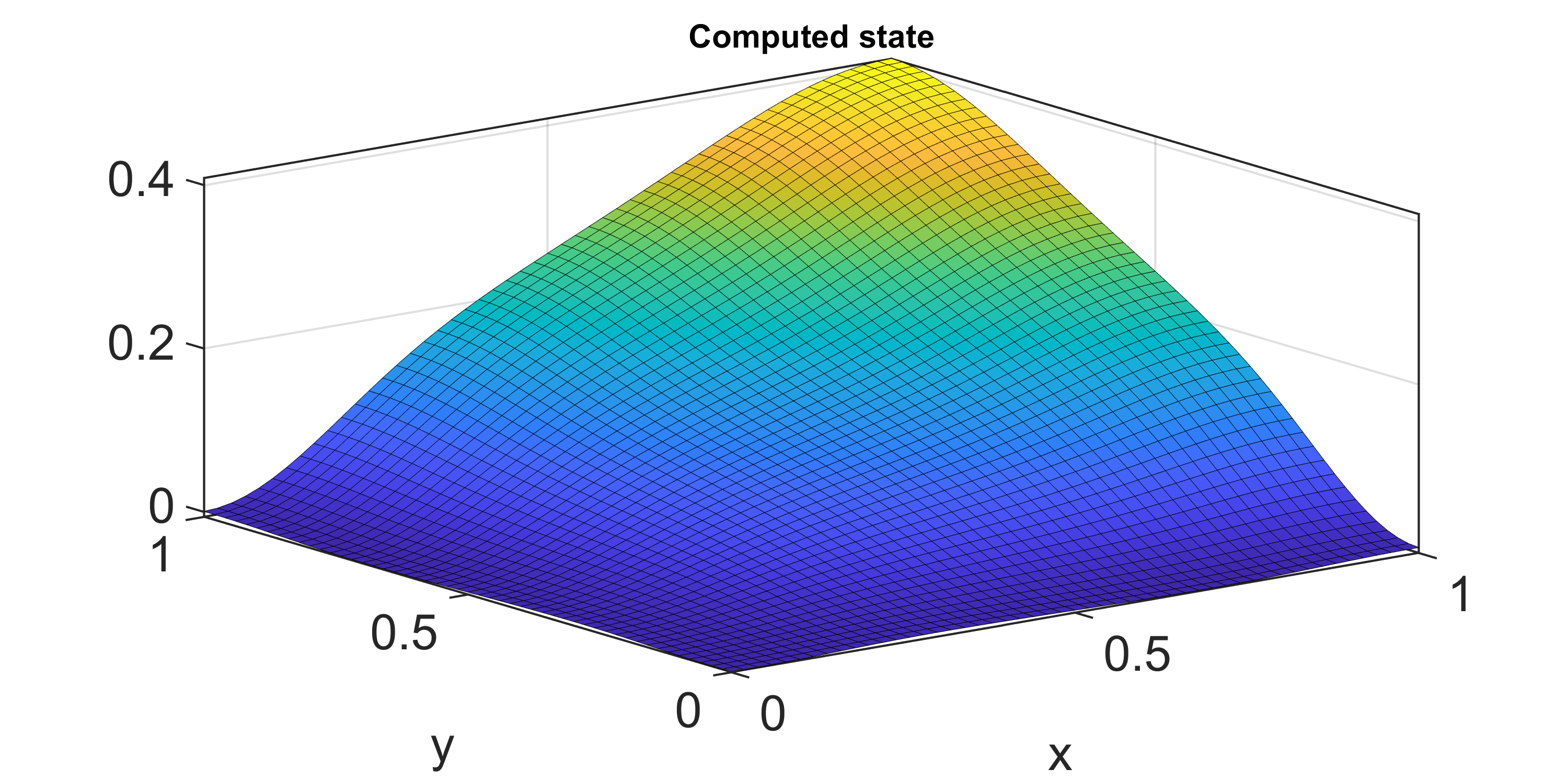} }}
    \caption{ On the left : the desired state; On the right : the computed state; for $\epsilon=0.08, \lambda=0.001$.}
    \label{s3_2d_state}
\end{figure} 
\begin{figure}[h!]
    \centering
    \subfloat{{\includegraphics[height=5.5cm,width=7cm]{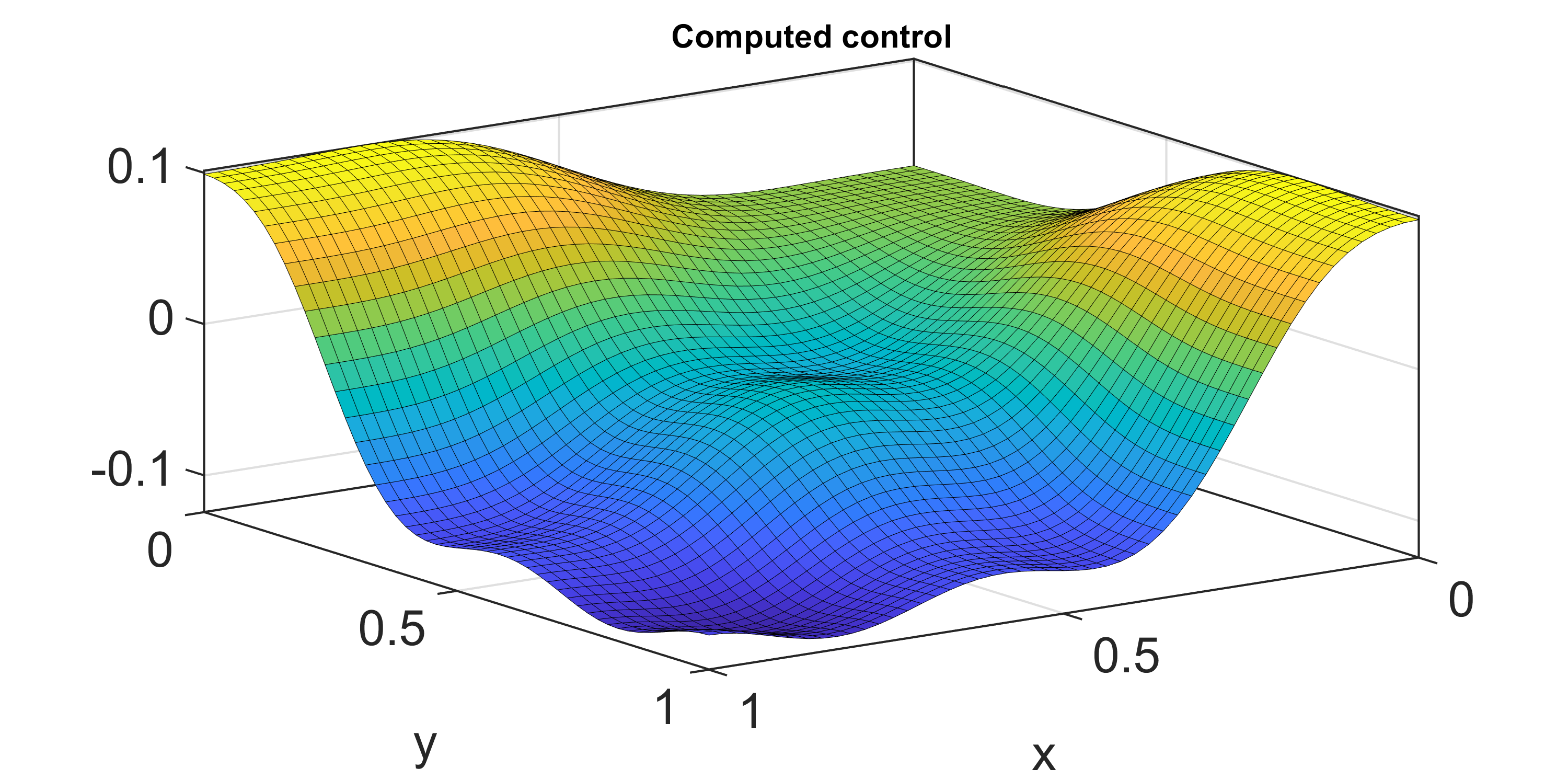} }}
     \subfloat{{\includegraphics[height=5.5cm,width=7cm]{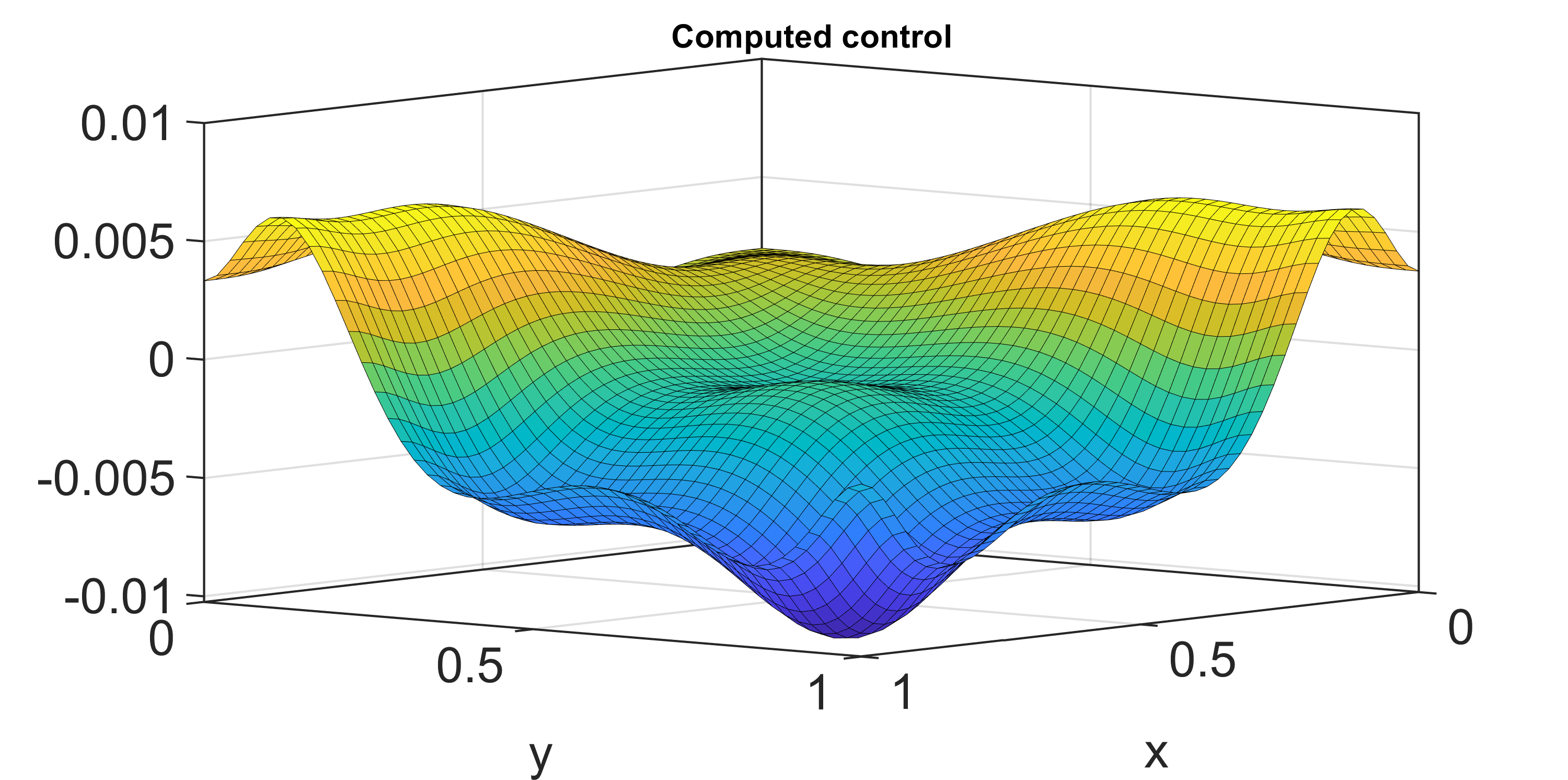} }}
    \caption{ On the left : computed control at $t=0$; On the right :  computed control at $t=T-\delta_t$; for $\epsilon=0.08, \lambda=0.001$.}
    \label{s3_2d_control}
\end{figure}
The final target state and calculated state at final time are shown in Figure \ref{s3_2d_state}. It is obvious from the plot that the computed state is close to the desired state. We show the corresponding control at $t=0$ and $t=T-\delta_t$ in Figure \ref{s3_2d_control}.

\section{Conclusions}
In this work, we formulate linear and nonlinear discretization schemes for the OCP with the CH equation as constrain. We present rigorous convergence analysis for all the proposed schemes. Lastly we verify the numerical accuracy and present numerical solution for various desired state in 1D and 2D.
\section*{Acknowledgement} The authors would like to thank the CSIR (File No : 09/1059(0019)/2018-EMR-I) and DST-SERB (File No : SRG/2019/002164) for the research grant and IIT Bhubaneswar for providing excellent research environment.

\bibliographystyle{siam}
\bibliography{OCPbib}

\end{document}